\documentclass{article}

\usepackage[left=1in, right=1in, top=1in, bottom=1in]{geometry}

\usepackage[square,sort,comma,numbers]{natbib}

\RequirePackage{algorithm}
\RequirePackage{algorithmic}

\usepackage{amsmath}
\usepackage{amssymb}
\usepackage{mathtools}
\usepackage{amsthm}
\usepackage{bm}

\usepackage{threeparttable}
\usepackage{booktabs}

\usepackage[capitalize,noabbrev]{cleveref}

\theoremstyle{plain}
\newtheorem{theorem}{Theorem}[section]
\newtheorem{proposition}[theorem]{Proposition}

\theoremstyle{definition}

\theoremstyle{remark}
\newtheorem{remark}[theorem]{Remark}

\usepackage[textsize=tiny]{todonotes}

\def\BW#1{\textbf{\textcolor{red}{#1}}}

\def\va{{\bm{a}}}
\def\vb{{\bm{b}}}

\def\vf{{\bm{f}}}
\def\vg{{\bm{g}}}
\def\vh{{\bm{h}}}

\def\vk{{\bm{k}}}

\def\vp{{\bm{p}}}

\def\vs{{\bm{s}}}

\def\vu{{\bm{u}}}

\def\vw{{\bm{w}}}
\def\vx{{\bm{x}}}
\def\vy{{\bm{y}}}
\def\vz{{\bm{z}}}

\def\mA{{\bm{A}}}

\def\mI{{\bm{I}}}

\def\mL{{\bm{L}}}

\def\mW{{\bm{W}}}

\def \RR {{\mathbb{R}}}

\usepackage{verbatim}

\usepackage{wrapfig}

\title{Proximal Implicit ODE Solvers for Accelerating Learning Neural ODEs}

{\footnotesize\author{
  Justin Baker$^*$ \\
  Department of Mathematics\\ 
  University of Utah\\ 
  \and
  Hedi Xia\footnote{Co-first author.} \\
 Department of Mathematics\\
  University of California, Los Angeles\\
  \and
  Yiwei Wang \\
  Departmemt of Applied Mathematics\\
  Illinois Institute of Technology\\
  \and
  Elena Cherkaev,\ Akil Narayan \\
  Department of Mathematics\\
  University of Utah\\
  \and
  Long Chen,\ Jack Xin \\
  Department of Mathematics\\
  University of California, Irvine\\
  \and
  Andrea L. Bertozzi,\ Stanley J. Osher \\
  Department of Mathematics\\
  University of California, Los Angeles\\
  \and
Bao Wang\footnote{Please correspond to: wangbaonj@gmail.com} \\
Department of Mathematics\\
  Scientific Computing and Imaging (SCI) Institute\\
  University of Utah\\ 
}}

\begin{document}

\maketitle

\begin{abstract}
Learning neural ODEs often requires solving very stiff ODE systems, primarily using explicit adaptive step size ODE solvers. These solvers are computationally expensive, requiring the use of tiny step sizes for numerical stability and accuracy guarantees. This paper considers learning neural ODEs using implicit ODE solvers of different orders leveraging proximal operators. The proximal implicit solver consists of inner-outer iterations: the inner iterations approximate each implicit update step using a fast optimization algorithm, and the outer iterations solve the ODE system over time. The proximal implicit ODE solver guarantees superiority over explicit solvers in numerical stability and computational efficiency. We validate the advantages of proximal implicit solvers over existing popular neural ODE solvers on various challenging benchmark tasks, including learning continuous-depth graph neural networks and continuous normalizing flows.
\end{abstract}

\section{Introduction}\label{sec:intro}
Neural ODEs \cite{chen2018neural} are a class of continuous-depth neural networks \cite{rosenblatt1961principles,cohen1983absolute} that are particularly suitable for learning complex dynamics from irregularly-sampled sequential data, see, e.g., \cite{chen2018neural,latentODE,NEURIPS2019_21be9a4b,massaroli2020dissecting,norcliffe2020_sonode}. Neural ODEs are used in many  applications,
including image classification and image generation \cite{chen2018neural}, learning dynamical systems \cite{latentODE}, modeling probabilistic distributions of complex data by a change of variables \cite{grathwohl2018scalable,pointflow,jiang2020shapeflow}, and scientific computing \cite{dutta2021neural,baker2022learning}.  
Recently, neural ODEs have also been employed for building continuous-depth graph neural networks (GNNs), achieving remarkable results for deep graph learning \cite{poli2019graph,pmlr-v119-xhonneux20a,
pmlr-v139-chamberlain21a,thorpe2022grand}. Mathematically, a neural ODE can be formulated by the following first-order ODE:
\begin{equation}\label{eq:NODE}
\frac{d{\vh}(t)}{dt}= {\bm f}({\vh}(t),t,\theta),
\end{equation}
where ${\bm f}({\vh}(t),t,\theta) \in \RR^d$ is specified by a neural network parameterized by $\theta$, e.g., a two-layer feed-forward neural network. Starting from the input ${\vh}(0)$, neural ODEs learn the representation of the input 
and perform prediction by solving \eqref{eq:NODE} from the initial time $t=0$ to the terminal time $T$ using a numerical ODE solver with a given error tolerance, often with adaptive step size solver or adaptive solver for short \cite{DORMAND198019}. Solving \eqref{eq:NODE} from $t=0$ to $T$ in a single pass with an adaptive solver requires evaluating $\vf({\vh}(t),t,\theta)$ at various timestamps, with the computational complexity counted by the number of forward function evaluations (forward NFEs), which is nearly proportional to the computational time, see \cite{chen2018neural} for details.

The \emph{adjoint sensitivity method}, or the adjoint method \cite{adjoint}, is a memory-efficient algorithm for training neural ODEs. If we regard the output ${\vh}(T)$ as the prediction and denote the loss between ${\vh}(T)$ and the ground truth as $\mathcal{L}:=\mathcal{L}(\vh(T))$. Let ${\va}(t):={\partial \mathcal{L}}/{\partial {\vh}(t)}$ be the adjoint state, then we have (see e.g., \cite{chen2018neural,adjoint} for details)
\begin{equation}\label{eq:NODE:gradient}
\frac{d\mathcal{L}}{d\theta} = \int_0^T{\va}(t)^\top \frac{\partial \vf({\vh}(t),t,\theta)}{\partial\theta} dt,
\end{equation}
with ${\va}(t)$ satisfying the following adjoint ODE
\begin{equation}\label{eq:NODE:adjoint}
\frac{d{\va}(t)}{dt} = -{\va}(t)^\top \frac{\partial}{\partial {\vh}} \vf({\vh}(t),t,\theta),
\end{equation}
which is solved numerically from $t=T$ to $0$ and also requires the evaluation of the right-hand side of \eqref{eq:NODE:adjoint} at various timestamps, with the computational complexity, or number of evaluation of the function $-{\va}(t)^\top \frac{\partial}{\partial {\vh}} \vf({\vh}(t),t,\theta)$, measured by the backward NFEs. 

\subsection{Computational bottlenecks of Neural ODEs}
\begin{figure}[!ht]
\centering
\begin{tabular}{cc}
\hskip -0.2cm
\includegraphics[clip, trim=0.01cm 0.01cm 0.01cm 0.01cm, width=0.49\columnwidth]{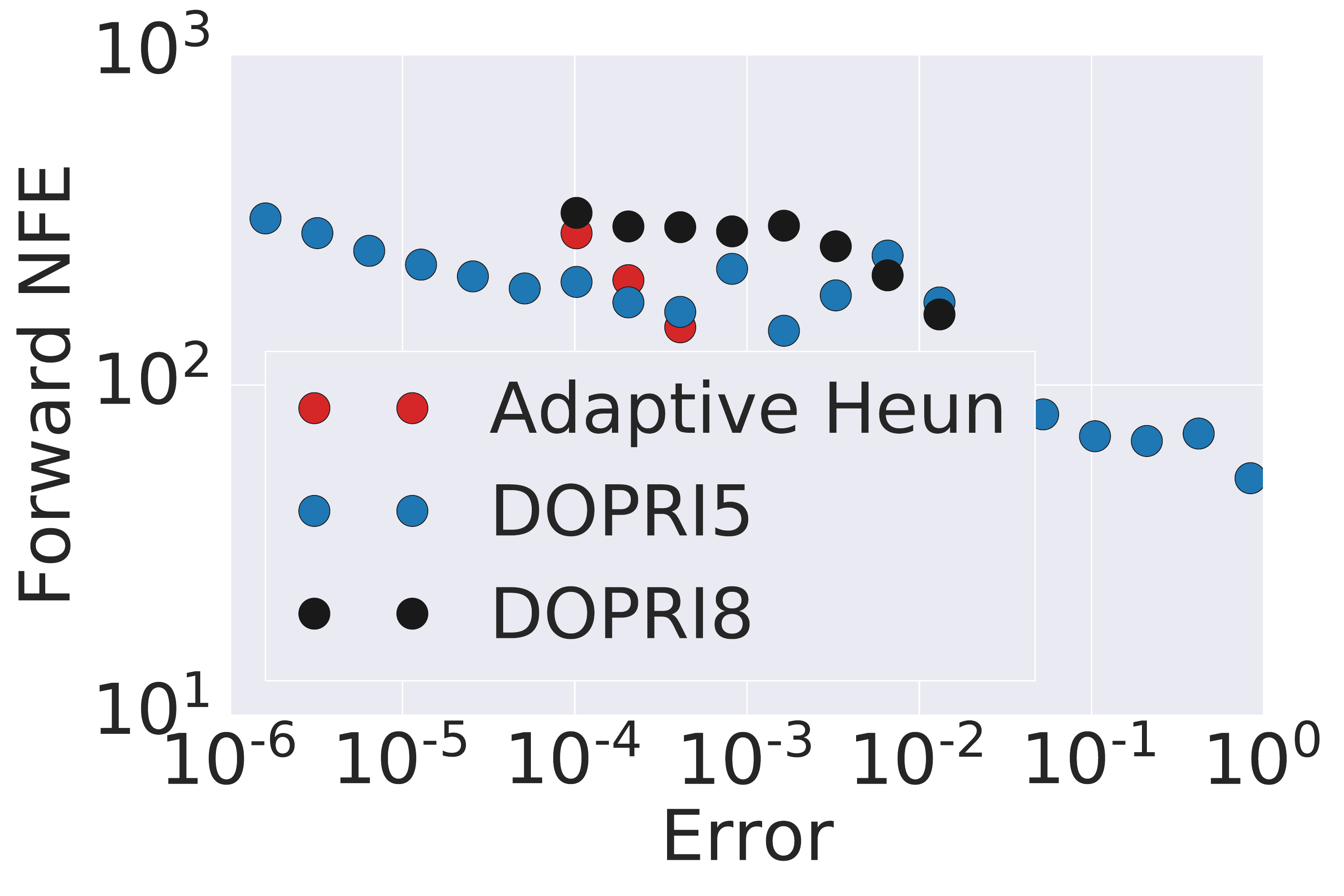}&
\hskip -0.4cm
\includegraphics[clip, trim=0.01cm 0.01cm 0.01cm 0.01cm, width=0.49\columnwidth]{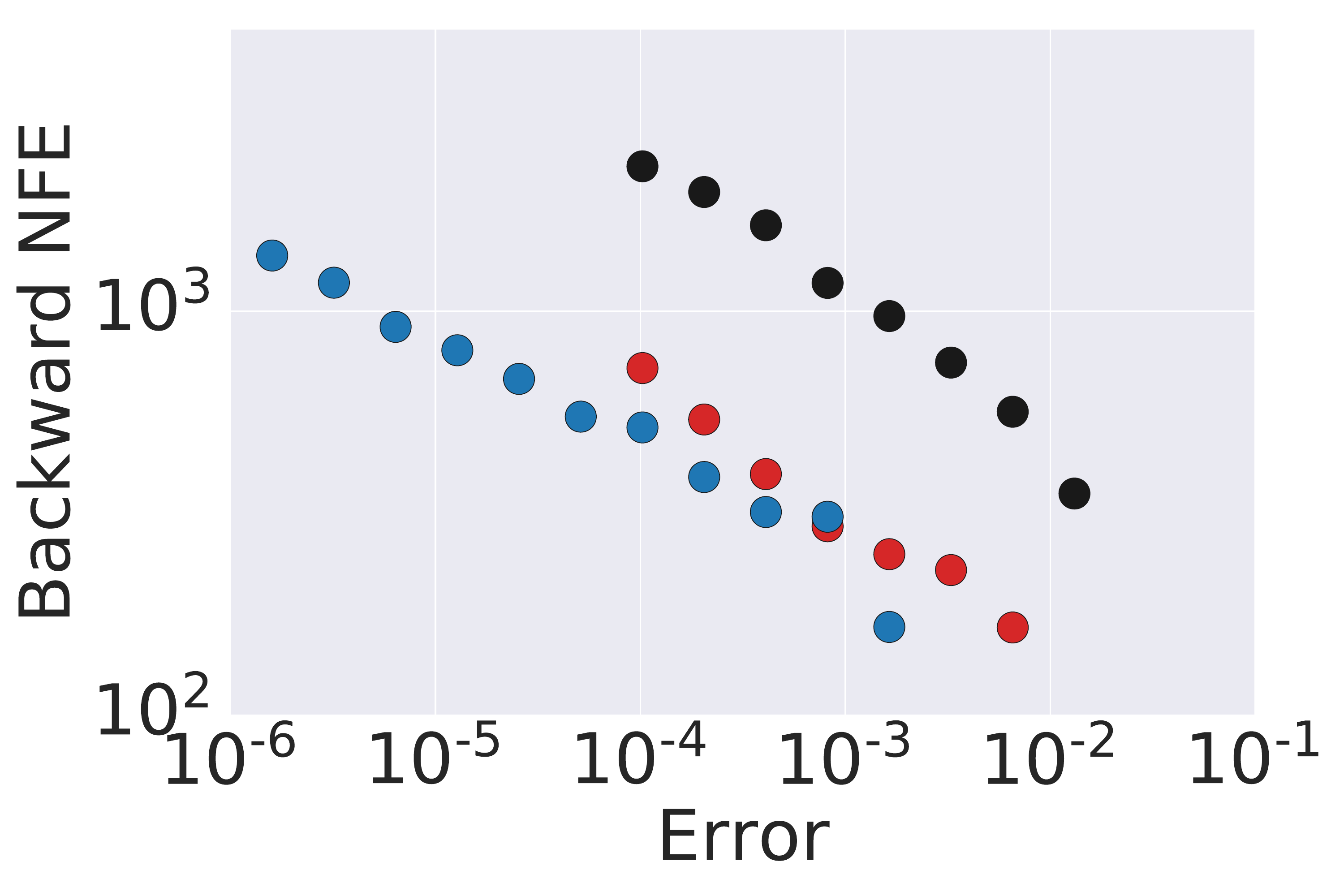}\\
\end{tabular}
\caption{Error tolerance vs. forward and backward NFEs of different adaptive solvers for training the GRAND model for CoauthorCS graph node classification.}
\label{fig:GRAND-NFEs}
\end{figure}
The computational bottlenecks of solving the neural ODE \eqref{eq:NODE} and its adjoint ODE \eqref{eq:NODE:adjoint} come from the numerical stability and numerical accuracy guarantees. Since $\vf(\vh(t),t,\theta)$ is usually high dimensional; direct application of implicit ODE solvers requires solving a system of high dimensional nonlinear equations, which is computationally inefficient, see Section~\ref{subsec:diffusion-1D} for an illustration. As such, explicit solvers --- especially the explicit adaptive solvers, e.g., the Dormand-Prince method \cite{DORMAND198019} --- are the current default numerical solvers for neural ODE training, test, and inference; see Appendix~\ref{sec:explicit-solvers} for a brief review of explicit adaptive neural ODE solvers. On the one hand, to guarantee numerical accuracy, we have to employ very small step sizes to discretize the neural ODE and its adjoint ODE. On the other hand, both ODEs involved in learning neural ODEs are usually very stiff, which further constrains the step size. 
Intuitively, a stiff ODE system has very different eigenvalues 
that reveal dramatically different scales of dynamics. To resolve these dynamics accurately, we have to accommodate for dynamics at significantly different scales, resulting in the use of very small step sizes. To demonstrate this issue, we consider training a recently proposed neural ODE-based GNN, named graph neural diffusion (GRAND) \cite{pmlr-v139-chamberlain21a}, for the benchmark CoauthorCS graph node  
classification; the detailed experimental settings can be found in Section~\ref{subsec:GRAND}. GRAND parametrizes the right-hand side of \eqref{eq:NODE} with the Laplacian operator, resulting in a class of diffusion models. The diffusion model is very stiff from the numerical ODE viewpoint since the ratio between the magnitude of the largest and smallest eigenvalues is infinite. 
Figure~\ref{fig:GRAND-NFEs} plots the error tolerance of the adaptive solver vs. forward and backward NFEs of different adaptive solvers --- including adaptive Heun \cite{suli2003introduction} and two Dormand-Prince methods \cite{DORMAND198019} (DOPRI5 and DOPRI8) --- for training GRAND for CoauthorCS node classification. We see that 1) all three adaptive solvers require significant NFEs in solving both neural ODE and its adjoint ODE; as the error tolerance reduces both forward and backward NFEs increase very rapidly. And 2) {the} higher-order scheme, e.g., DOPRI8 requires more NFEs than the lower-order scheme, indicating that the stability is a dominating factor in choosing the step size. Moreover, it is worth mentioning that both DOPRI5 and DOPRI8 often fail to solve the adjoint ODE when a large error tolerance is used, based on our experiments. 












The fundamental issues of learning stiff neural ODEs using explicit solvers due to stability concerns and the high computational cost of direct application of implicit solvers motivate us to study the following question: 

\emph{Can we modify or relax the implicit solvers to make them suitable for learning neural ODEs with significantly reduced computational costs than existing benchmark solvers?}

\subsection{Our contribution}

We answer the above question affirmatively by considering learning neural ODE-style models using the proxy of a few celebrated implicit ODE solvers --- including backward Euler, backward differentiation formulas (BDFs), and Crank-Nicolson --- leveraging the proximal operator \cite{parikh2014proximal}. These proximal solvers reformulate implicit ODE solvers as variational problems. Each solver contains inner-outer iterations, where the inner iterations approximate a one-step update of the implicit solver, and the outer iterations solve the ODE over time. Leveraging fast optimization algorithms for solving the inner optimization problem, these proximal solvers are remarkably faster than explicit adaptive ODE solvers in learning stiff neural ODEs. 
We summarize the major benefits of proximal solvers below:
\begin{itemize}
\item Due to the implicit nature of proximal algorithms, they are much less encumbered by the numerical stability issue than explicit solvers. 
Therefore, the proximal algorithms allow the use of very large step sizes for neural ODEs training, test, and inference. 

\item To achieve the same numerical accuracy in solving stiff neural ODEs, proximal algorithms can save 
significant NFEs in both forward and backward propagation. 

\item Training neural ODE-style models 
using proximal algorithms maintains 
the generalization accuracy of the model compared to using benchmark adaptive solvers.  
\end{itemize}







\subsection{More related works}
In this part, we discuss several related 
works in three directions: reducing NFEs in learning neural ODEs, advances of proximal algorithms, and the development of algorithms for learning neural ODEs. 


\paragraph{Reducing NFEs in learning neural ODEs by model design and regularization.}
Several algorithms have been developed to reduce the NFEs for learning neural ODEs. They can be classified into three categories: 1. Improving the ODE model and neural network architecture, and notable works in this direction include augmented neural ODEs \cite{NEURIPS2019_21be9a4b}, higher-order neural ODEs \cite{norcliffe2020_sonode}, heavy-ball neural ODEs \cite{HBNODE:2021}, and neural ODEs with depth variance \cite{massaroli2020dissecting}. These models can reduce the forward NFEs significantly, and heavy-ball neural ODEs can also reduce the backward NFEs remarkably. 2. Learning neural ODEs with regularization, including weight decay \cite{grathwohl2018scalable}, regularizing ODE solvers and learning dynamics \cite{pmlr-v119-finlay20a,NEURIPS2020_2e255d2d,NEURIPS2020_f1686b4b,NEURIPS2020_a9e18cb5,pmlr-v139-pal21a}. And 3. Input augmentation \cite{NEURIPS2019_21be9a4b} and data control \cite{massaroli2020dissecting}. Our work focuses on accelerating learning neural ODEs using the proximal form of implicit solvers.


\paragraph{Advances of proximal algorithms.} Proximal algorithms have been the workhorse for solving nonsmooth, large-scale, or distributed optimization problems \cite{parikh2014proximal}. The core matter is the evaluation of proximal operators \cite{bauschke2011convex}. The proximal algorithms are widely used in statistical computing and machine learning \cite{polson2015proximal}, 
channel pruning of neural 
networks \cite{yang2019channel,dinh2020sparsity}, image processing \cite{beck2009fast}, matrix completion \cite{marjanovic2012l_q,ZYX_17}, computational optimal transport \cite{salim2020wasserstein,peyre2019computational}, game theory and optimal control \cite{attouch2008alternating}, etc.
Proximal operators can be viewed as 
backward Euler method, see Section~\ref{sec:proximal-solvers} for a brief mathematical introduction. From an implicit solver viewpoint, a proximal formulation of the backward Euler method has been applied to the stochastic gradient descent training of neural networks \cite{chaudhari2019entropy,chaudhari2018deep}.





\paragraph{Developments of efficient neural ODE learning algorithms.} 
Solving an ODE is required in both forward and backward propagation of learning neural ODEs. The default numerical solvers are explicit adaptive Runge-Kutta schemes \cite{press1992adaptive,chen2018neural}, especially the Dormand-Prince method \cite{DORMAND198019}. To solve both forward and backward ODEs accurately, the adaptive solvers will evaluate the right-hand side of ODEs at many intermediate timestamps, causing tremendous computational burden. Checkpoint schemes have been proposed to reduce the computational cost \cite{gholami2019anode,pmlr-v119-zhuang20a},
often reducing computational cost by compromising memory efficiency. Recently, the symplectic adjoint method has been proposed \cite{matsubara2021symplectic}, which solves neural ODEs using a symplectic integrator and obtains the exact gradient (up to rounding error) with memory efficiency. Approximate gradients using interpolation instead of adjoint method \cite{NEURIPS2020_c24c6525} has also been used to accelerate learning neural ODEs. Our proposed proximal solvers can be integrated with these new adjoint methods for training neural ODEs. 





\subsection{Notation}
We denote scalars by lower or upper case letters; vectors and matrices by lower and upper case boldface letters, respectively. We denote the magnitude of a complex number $z$ as $|z|$. For a vector $\vx = (x_1, \ldots, x_d)^\top\in \mathbb{R}^d$, we use $\|\vx\| := {(\sum_{i=1}^d |x_i|^2)^{1/2}}$ 
to denote its $\ell_2$-norm.
We denote the vector whose entries are all 0s as $\mathbf{0}$. For two vectors $\va$ and $\vb$, we denote their inner product as $\langle\va,\vb\rangle$. For a matrix $\mA$, we use $\mA^\top$ and $\mA^{-1}$ 
to denote its transpose and inverse, 
respectively. 
We denote the identity matrix as $\mI$. For a function $f(\vx): \mathbb{R}^d \rightarrow \mathbb{R}$, we denote $\nabla f(\vx)$ and $\nabla^2f(\vx)$ as its gradient and Hessian, respectively. We denote $a=\mathcal{O}(b)$, if there is a constant $C$ such that $a\leq Cb$.

\subsection{Organization}
We organize the paper as follows: 
In Section~\ref{sec:proximal-solvers}, we present the proximal form of a few celebrated implicit ODE solvers for learning neural ODEs. 
We analyze the numerical stability and convergence of proximal solvers and contrast them with explicit solvers in Section~\ref{sec:error-analysis}. We verify the efficacy of proximal solvers in Section~\ref{sec:exp}, followed by concluding remarks. Technical proofs and more experimental details are provided in the appendix.

\section{Proximal Algorithms for Neural ODEs}\label{sec:proximal-solvers}

In this section, we formulate the proximal version 
of several celebrated implicit ODE solvers, including backward Euler, Crank-Nicolson, BDFs, and a class of single-step multi-stage schemes.


\subsection{A proximal viewpoint of the backward Euler solver}
We first consider solving a neural ODE \eqref{eq:NODE} using the proximal backward Euler solver. Directly discretizing 
\eqref{eq:NODE} using the backward Euler scheme with a constant step size $s$ gives
\begin{equation}\label{eq:backward-euler}
\vh_{k+1} = \vh_k + s{\bm f}(\vh_{k+1}).
\end{equation}
Equation~\eqref{eq:backward-euler} is a system of nonlinear equations, which can be solved efficiently by Newton's method. But when the dimension of $\vf$ is very high, solving \eqref{eq:backward-euler} can be computationally infeasible. Nevertheless, if ${\bm f}(\vz)$ is the gradient of a function $- F({\vz})$, we can reformulate the update in \eqref{eq:backward-euler} using the proximal operator as follows
\begin{equation}\label{eq:prox-backward-euler}
\begin{aligned}
\vh_{k+1} 
&= \arg\min_{\vz \in \mathcal{A}}\Big\{\frac{1}{2s}\|\vz-\vh_k\|_2^2 +  F(\vz) \Big\},
\end{aligned}
\end{equation}
where $\vf(\vz) = - \nabla F(\vz)$, and $\mathcal{A}$ is the admissible domain of $\vz$, which is considered to be the whole space in this work; we omit $\mathcal{A}$ in the following discussion.
To see why \eqref{eq:prox-backward-euler} is equivalent to \eqref{eq:backward-euler}, we note that $\vh_{k+1}$ is a stationary point of $G(\vz) = \frac{1}{2}\|\vz-\vh_k\|_2^2+s F(\vz)$, resulting in
$$
\frac{d}{d\vz}\Big(\frac{1}{2s}\|\vz-\vh_k\|_2^2 + F(\vz)\Big)\Big|_{\vz = \vh_{k+1}} = 0,
$$
replacing $\vz$ with $\vh_{k+1}$ in the above equation, we get the backward Euler solver \eqref{eq:backward-euler} (note $\nabla F(\vz) = -\vf(\vz)$).
Based on the proximal formulation of the backward Euler method \eqref{eq:backward-euler}, we can solve the neural ODE \eqref{eq:NODE} using an inner-outer iteration scheme. 
The outer iterations perform backward Euler iterations over time, and the inner iterations solve the optimization problem formulated in \eqref{eq:prox-backward-euler}. 
We summarize the inner-outer iteration scheme for the proximal backward Euler solver in Algorithm~\ref{alg:proximal-backward-euler}.

\begin{algorithm}
\caption{Proximal backward Euler for solving \eqref{eq:NODE}}
\begin{algorithmic}\label{alg:proximal-backward-euler}
\REQUIRE Step size $s>0$, inner iteration number $n$\\
\textbf{for}~$k=1,2,\ldots$ \\
~~ \textbf{step 1}: Let $\vz^0=\vh_k$\\
~~~ \textbf{step 2}: Start from $\vz^0$ solve inner minimization problem $$
\arg\min_\vz\Big\{\frac{1}{2s}\|\vz-\vh_k\|_2^2 +  F(\vz) \Big\},$$
~~~~~~~~ resulting in a sequence $\vz^1,\vz^2,\ldots,\vz^n$.\\
~~~ \textbf{step 3}: Let $\vh_{k+1}=\vz^n$. \\
\textbf{end for}\\
\end{algorithmic}
\end{algorithm}

\subsubsection{Solving the inner minimization problem}
Another piece of the recipe is how to effectively solve the inner minimization problem of 
proximal backward Euler \eqref{eq:prox-backward-euler}, i.e., the following optimization problem
\begin{equation}\label{eq:inner-problem}
\arg\min_\vz G(\vz) := \frac{1}{2s}\|\vz-\vh_k\|_2^2 + F(\vz).
\end{equation}
One of the simple yet efficient algorithms for solving \eqref{eq:inner-problem} is gradient descent (GD), which updates as follows: for $i=1,2,\ldots,n-1$, perform
\begin{equation}\label{eq:GD}
\vz^{i+1} = \vz^i - \eta \nabla G(\vz^i),\ \eta>0,
\end{equation}
where $\eta$ is the step size and $\nabla G(\vz^i)=(\vz^i-\vh_k)/s - \vf(\vz^i)$. Here, notice that we do not need to know the exact form of $F(\vz)$, and the GD update becomes
\begin{equation}\label{eq:GD2}
\vz^{i+1}=\vz^i-\eta\Bigg(\frac{\vz^i-\vh_k}{s}-\vf(\vz^i)\Bigg).
\end{equation}
There are various off-the-shelf acceleration schemes that exist to accelerate the convergence of \eqref{eq:GD2}, e.g., Nesterov accelerated gradient (NAG) \cite{nesterov1983method}, NAG with adaptive 
restart (Restart) \cite{roulet2017sharpness}, and GD with nonlinear conjugate gradient style momentum, e.g., Fletcher–Reeves momentum (FR) \cite{fletcher1964function,wang2020stochastic}. In particular, FR, which will be used in our experiments, updates $\vz^i$ as follows: 
\begin{equation}\label{eq:FR}
\begin{aligned}
\vp^i &= \Bigg(\frac{\vz^i-\vh_k}{s}-\vf(\vz^i)\Bigg) + \beta_i\vp^{i-1}\\
\vz^{i+1} &= \vz^i - \eta \vp^i,
\end{aligned}
\end{equation}
where $\vp_0={\bf 0}$ and the scalar $\beta_i$ is given below
\begin{equation}\label{eq:beta}
\beta_i=\begin{cases}
0 &\  \mbox{if}\ i=0,\\
\frac{\big({(\vz^i-\vh_k)}/{s}-\vf(\vz^i)\big)^\top \big({(\vz^i-\vh_k)}/{s}-\vf(\vz^i)\big)}{ 
\big({(\vz^{i-1}-\vh_k)}/{s}-\vf(\vz^{i-1})\big)^\top \big({(\vz^{i-1}-\vh_k)}/{s}-\vf(\vz^{i-1})\big)} &\ \mbox{if}\ i\geq 1.
\end{cases}
\end{equation}

With the access to the gradient $-\vf(\vz)$ we can also employ L-BFGS \cite{liu1989limited} to solve the inner minimization problem \eqref{eq:inner-problem}. Based on our testing, 
{GD with FR momentum usually outperforms the 
other methods listed above}, see Section~\ref{subsec:diffusion-1D} for a comparison of different optimization algorithms for solving the one-dimensional diffusion problem.



\begin{remark}
The above discussion of gradient-based optimization algorithms assumes the existence of $F(\vz)$ whose gradient is $-\vf(\vz)$. When such a function $F(\vz)$ does not exist, we can still use the iteration \eqref{eq:GD2} to solve the inner minimization problem, which we can regard as a fixed point iteration. Moreover, we can accelerate the convergence of \eqref{eq:GD2} using the Anderson acceleration \cite{anderson1965iterative,walker2011anderson}, and we leave it as future work. Based on our numerical tests, gradient-based optimization algorithms work quite well in solving the inner optimization problem \eqref{eq:inner-problem} across all studied benchmark tasks.
\end{remark}

\paragraph{Stopping criterion of inner solvers.}
Given an error tolerance $\epsilon$ of the inner optimization solver, we stop the inner iteration if $\|\vz^{i+1}-\vz^i\|\leq \epsilon$.

\subsubsection{Explicit vs. proximal solvers}
\begin{figure}[!ht]
\centering
\begin{tabular}{cc}
\hskip -0.2cm
\includegraphics[clip, trim=0.01cm 0.01cm 0.01cm 0.01cm, width=0.46\columnwidth]{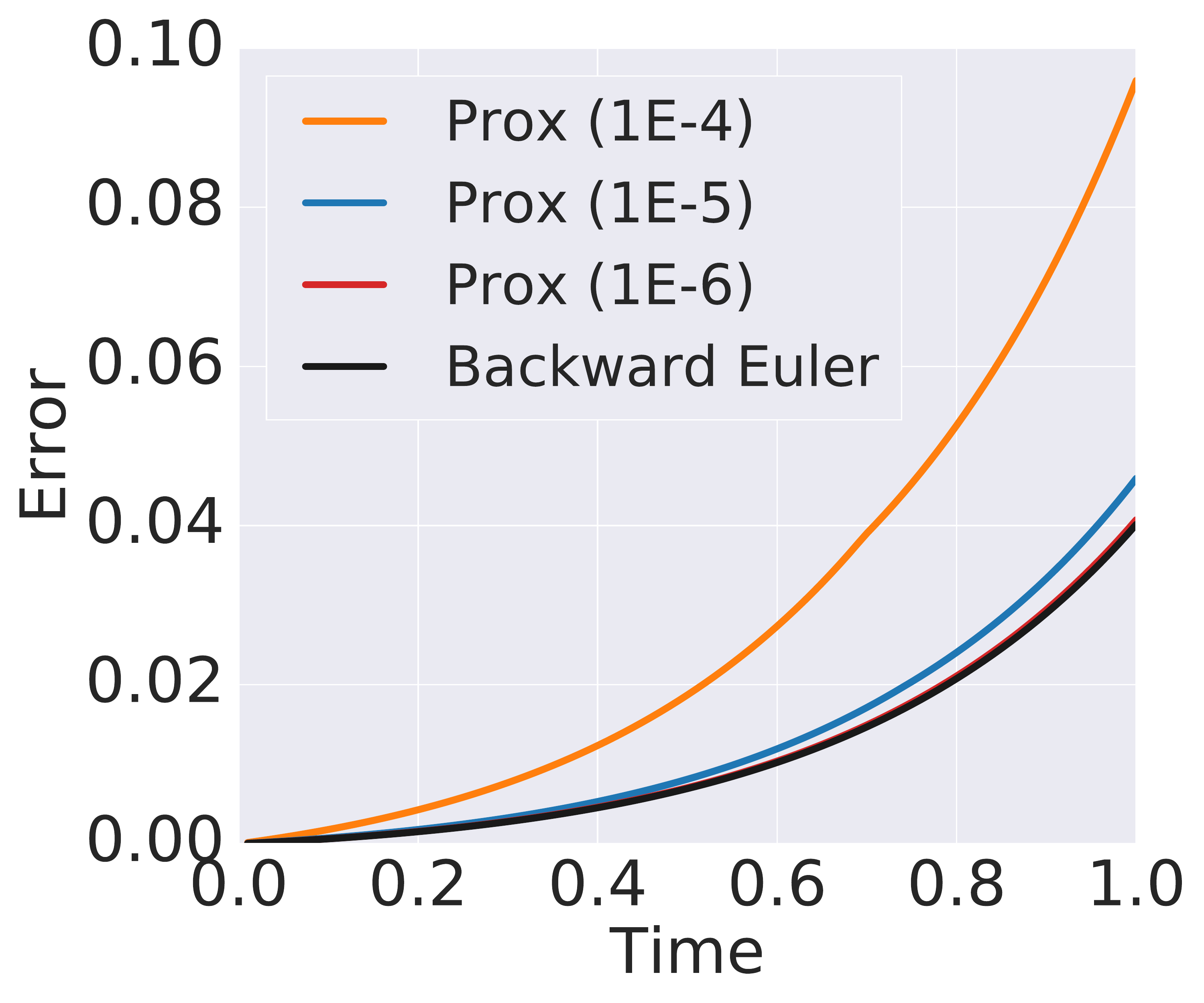}&
\hskip -0.4cm
\includegraphics[clip, trim=0.01cm 0.01cm 0.01cm 0.01cm, width=0.46\columnwidth]{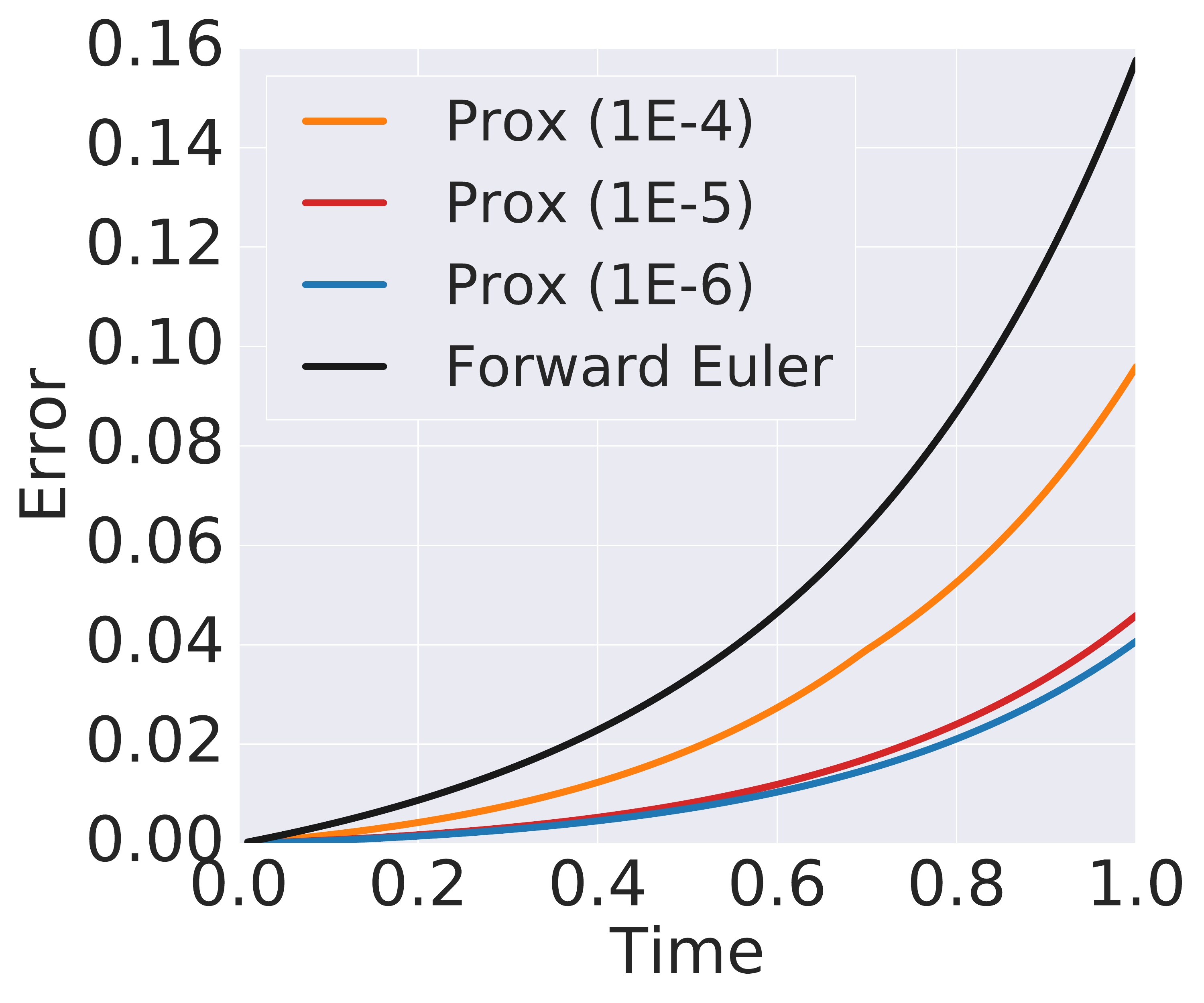}\\
(a) & (b) \\
\end{tabular}
\caption{(a) Time vs. numerical errors of the backward Euler method and the proximal backward Euler (Prox) with different inner error tolerances
for solving the one-dimensional ODE \eqref{eq:1D-benchmark}. As the error of the inner solver decreases, the proximal backward Euler approaches the backward Euler. (b) Comparison of proximal backward Euler using different inner solver accuracy against the forward Euler for solving the same problem in (a). We see that the proximal backward Euler method accumulates error more slowly than the forward Euler scheme.}
\label{fig:prox_impl_err}
\end{figure}

Before presenting more proximal implicit solvers, we compare forward Euler, backward Euler, and the proximal backward Euler for solving the following benchmark one-dimensional ODE, which comes from \cite{atkinson2011numerical}
\begin{equation}\label{eq:1D-benchmark}
\frac{dh(t)}{dt}=3h(t)-2\cos(t)-4\sin(t)
\end{equation}
with initial condition $h(0)=1$. We use the integration time step size $0.01$ for all the above ODE solvers.
Figure~\ref{fig:prox_impl_err} (a) plots the 
error between numerical and analytic solutions of proximal backward Euler with different inner solver tolerances and the backward Euler solver. 
It is evident that the proximal backward Euler approximates backward Euler quite well, and the approximation becomes more accurate as the accuracy of the inner solver enhances. Figure~\ref{fig:prox_impl_err} (b) contrasts the proximal backward Euler with the forward Euler solver. We see that with an appropriate tolerance of the inner solver, proximal backward Euler accumulates error slower than the forward Euler solver as the time increases. These advantages of the proximal backward Euler solver shed light on improving learning stiff neural ODEs and alleviating error accumulation.

\subsection{High-order implicit schemes via proximal operators}
A natural extension of the backward Euler scheme is the {single-step}, multi-stage schemes whose proximal form are proposed in \cite{zaitzeff2020variational}, and we formulated those schemes in Algorithm~\ref{alg:proximal-multi-stage}.
\begin{algorithm}
\caption{Proximal single-step, multi-stage scheme for solving  \eqref{eq:NODE}}
\begin{algorithmic}\label{alg:proximal-multi-stage}
\REQUIRE Step size $s>0$, stage $M$
\\
\textbf{for}~$k=1,2,\ldots$ \\
~~ \textbf{step 1}: Let $\vz^0=\vh_k$\\
~~~ \textbf{for}~$m=1,\ldots,M$ solve the inner problem
~~~ $$
\vz^{m} = \mathop{\arg \min}_{\vz} \Bigg\{ \sum_{i=0}^{m-1} \frac{\gamma_{m, i}}{2 s} \|{\vz} - {\vz}^{i}\|^2 + F({\vz})\Bigg\},$$
~~~~~~~~ resulting in a sequence $\vz^1,\vz^2,\ldots,\vz^M$.\\
~~~ \textbf{step 3}: Let $\vh_{k+1}=\vz^M$. \\
\textbf{end for}\\
\end{algorithmic}
\end{algorithm}

Clearly, the 1-stage scheme is the backward Euler scheme. The authors of \cite{zaitzeff2020variational} give conditions on the coefficient $\gamma_{m, i}$. However, it is quite tedious to get those coefficients. For example, a second-order unconditionally stable scheme can be constructed by taking the matrix $\boldsymbol{\Lambda}$ --- whose entries are $\{\gamma_{m,i}\}_{m,i=1}^3$ that appeared in Algorithm~\ref{alg:proximal-multi-stage} --- as follows
\begin{equation*}
\left(\begin{array}{ccc}
\gamma_{1,0} & 0 & 0 \\
\gamma_{2,0} & \gamma_{2,1} & 0 \\
\gamma_{3,0} & \gamma_{3,1} & \gamma_{3,2}
\end{array}\right) = 
\left(\begin{array}{ccc}
5 & 0 & 0 \\
-2 & 6 & 0 \\
-2 & \frac{3}{14} & \frac{44}{7}
\end{array}\right).
\end{equation*}
Furthermore, the matrix $\boldsymbol{\Lambda}$ of a third-order scheme is 
\begin{equation*}
\begin{pmatrix}
11.17 & 0 & 0 & 0 & 0 & 0 \\
-7.5 & 19.43 & 0 & 0 & 0 & 0 \\
-1.05 & -4.75 & 13.98 & 0 & 0 & 0 \\
1.8 & 0.05 & -7.83 & 13.8 & 0 & 0 \\
6.2 & -7.17 & -1.33 & 1.63 & 11.52 & 0 \\
-2.83 & 4.69 & 2.46 & -11.55 & 6.68 & 11.95
\end{pmatrix}
\end{equation*}

The main drawback of these multi-stage algorithms for training neural ODEs lies in the computational inefficiency. Note that for an $M$-stage scheme, one need to solve $M$ optimization problems in each iteration. 



\subsection{Proximal form of Crank-Nicolson}
Another {single-step}, second-order extension for the backward Euler 
is the  Crank-Nicolson scheme, given by
\begin{equation*}
 \vh_{k+1} = \vh_{k} + \frac{s}{2} ({\bm f} \left( \vh_{k+1} \right) + {\bm f} \left( \vh_{k} \right) ).
\end{equation*}
The Crank-Nicolson scheme is also known as an implicit Runge–Kutta method or  trapezoidal rule, which can be formulated in the following proximal form \cite{du2020phase}
\begin{equation}\label{eq:prox-cn}
\begin{aligned}\small
  {\vh}_{k+1} = \mathop{\arg\min}_{\vz} \Bigg\{& \frac{\|{\vz} - {\vh}_k\|^2}{s} + F({\vz})
  + \langle {\vz} - {\vh}_{k}, \nabla F({\vh}_k) \rangle \Bigg\}
\end{aligned}
\end{equation}

\subsection{The proximal viewpoint of BDF methods}
The 
backward Euler scheme has first-order accuracy, i.e., discretizing the ODE \eqref{eq:NODE} using \eqref{eq:backward-euler} with step size $s$ has error 
$\mathcal{O}(s)$. 
Backward differentiation formulas (BDFs) are higher-order multi-step 
methods that can be formulated as follows
\begin{equation*}
    \frac{a_s \vh_{k+1} - \mA_s(\vh_k, \vh_{k-1}, \ldots)}{s} = - \vf(\vh_{k+1}),
\end{equation*}
where $a_s$ is a constant, and $\mA_s$ is an operator acts on $\vh_k,\vh_{k-1},\ldots$. 
The backward Euler method corresponds to the case when $s = 1$, and $a_s = 1$ and $A_1(\vh_k) = \vh_k$. The second-order BDF2 scheme can be constructed by taking $a_s = \frac{3}{2}$ and  $\mA_2(\vh_k, \vh_{k-1}) = 2 \vh_k - \frac{1}{2} \vh_{k-1}$. 
Furthermore, the BDF2 scheme can be written in the following proximal form \cite{matthes2019variational,du2020phase}
\begin{equation*}\small
\vh_{k+1} = \mathop{\arg \min}_{\vz} \left\{ \frac{\|{\vz} - \vh_k\|^2}{s}  -  \frac{\|\vz - \vh_{k-1}\|^2}{4 s}+ F({\vz}) \right \},
\end{equation*}
if there exists $F(\vz)$ such that ${\bm f}(\vz) = - \nabla F(\vz)$.

We can further construct high-order BDFs, for instance the third- and fourth-order BDFs are given below
\begin{itemize}
\item BDF3:
\begin{equation*}\small
a_3 = \frac{11}{6}\ \ \mbox{and}\ \  \mA_3(\cdot)= 3 \vh_k - \frac{3}{2} \vh_{k-1} + \frac{1}{3} \vh_{k-2}.
\end{equation*}
\item BDF4:
\begin{equation*}\small
a_4 = \frac{25}{12}\ \ \mbox{and}\ \  \mA_4(\cdot) = 4 \vh_k - 3 \vh_{k-1} + \frac{4}{3} \vh_{k-2} - \frac{1}{4} \vh_{k-3}.
\end{equation*}
\end{itemize}
The corresponding proximal form of BDF3 and BDF4 are given as follows
\begin{itemize}
\item BDF3:
\begin{equation*}
\begin{aligned}
    \vh_{k+1} = & \mathop{\arg \min}_{\vz}  \left\{  \frac{3 \|{\vz} - \vh_k\|^2}{2 s}  -  \frac{3 \|\vz - \vh_{k-1}\|^2}{4 s} \right.
     \left. + \frac{\|\vz - \vh_{k-2}\|^2}{6 s} + F({\vz}) \right \},
\end{aligned}
\end{equation*}

\item BDF4:
\begin{equation*}
\begin{aligned}
\hspace{-0.2cm}    \vh_{k+1} &=  \mathop{\arg \min}_{\vz}  \left\{  \frac{2 \|{\vz} - \vh_k\|^2}{s}  -  \frac{3 \|\vz - \vh_{k-1}\|^2}{2 s} \right. 
     \left. + \frac{2 \|\vz - \vh_{k-2}\|^2}{3 s} - \frac{\|\vz - \vh_{k-2}\|^2}{8 s}  + F({\vz}) \right \},
\end{aligned}
\end{equation*}
\end{itemize}

{
Only one optimization problem needs to be solved in BDFs, which is the advantage of BDF methods over many other high-order implicit schemes.
}





\section{Stability and Convergence Analysis}\label{sec:error-analysis}
In this section, 
we will 1) compare the linear stability of different explicit and implicit solvers, 2) show that the error of the inner solver will not affect the stability of proximal solvers, and 3) analyze the convergence of proximal solvers.







\subsection{Linear stability: Implicit vs. explicit solvers}

It is well known that explicit solvers often require a small step size for numerical stability and convergence guarantees, especially for solving stiff ODEs. Consider the 
linear ODE
\begin{equation}\label{eq:linearODE}
    \vh' = \mA \vh, \quad \mA \in \mathbb{R}^{d \times d}.
\end{equation}
Assume $\mA$ has spectrum $\sigma(\mA) = \{ \lambda_j\}_{j=1}^d$ and ${\rm Re}(\lambda_j) < 0$ for $j = 1, \ldots d$, one can define the stiffness ratio as
\begin{equation*}
    Q = \frac{\max_{\lambda \in \sigma(\mA)} |{\rm Re}(\lambda)|}{\min_{\lambda \in \sigma(\mA)} |{\rm Re}(\lambda)|}
\end{equation*}
We say the ODE is stiff if $Q \gg 1$.
The {\bf linear stability domain} $\mathcal{D}$ of the underlying numerical method is the set of all numbers $z:=s \lambda_j$ for $j=1,\ldots,d$, such that $\lim_{k \rightarrow \infty} \vh_{k} = {\bf 0}$, where $\vh_k$ is the numerical solution of \eqref{eq:linearODE}  at the $k$-th step. Table. \ref{tab:linear stability} lists the linear stability domain of some single-step numerical schemes described above.

\begin{table}[!ht]
\begin{center}
\begin{tabular}{c|c}\hline
         \bf 
        \qquad \qquad \qquad Numerical methods \qquad \qquad \qquad 
         &  \qquad \qquad \qquad \bf Linear stability domain \qquad \qquad  \qquad \\
         \hline
         Forward Euler  & $\mathcal{D}_{\rm FE} = \{ z \in \mathbb{C} ~|~ |1 + z| < 1 \}$   \\
         Backward Euler & $\mathcal{D}_{\rm BE} = \{ z \in \mathbb{C} ~|~ |1 - z| > 1 \}$  \\
         Crank-Nicolson & $\mathcal{D}_{\rm CR} = \{ z \in \mathbb{C} ~|~ |\frac{1 + z/2}{1 - z / 2}| < 1 \}$   \\
         DOPRI5 & $\mathcal{D}_{\rm DP} = \{ z \in \mathbb{C} ~|~ | F_{1/2}(z)| < 1 \}$   \\
         \hline
    \end{tabular}
\end{center}
    \caption{Linear stability domains of several single step numerical ODE solvers. $|F_{1/2}(z)|<1$ stands for $|F_1(z)|<1$ and $|F_2(z)<1|$ where $F_1(z) = \sum_{r=0}^5 \frac{z^r}{r!} +\frac{z^{6}}{600}$ and $F_{2}(z)= \sum_{r=0}^4 \frac{z^r}{r!} + \frac{1097z^{5}}{120000}+\frac{161z^{6}}{120000}+\frac{z^{7}}{24000}$, see \cite{iserles2009first, DORMAND198019} for more details. }
    \label{tab:linear stability}
\end{table}

We call an ODE solver {\bf $A$-stable} if $\mathbb{C}^- := \{ z \in \mathbb{C} ~|~ {\rm Re }(z) < 0 \} \in \mathcal{D}$, where $\mathcal{D}$ is the linear stability domain of the solver. From Table~\ref{tab:linear stability},
it is clear that the backward Euler and Crank-Nicolson methods are both $A$-stable. For multi-step methods, e.g., BDFs,
there is no closed-form expression for the linear stability domain; see \cite{iserles2009first} for a detailed discussion. Among all BDFs, only BDF2 is $A$-stable. The $A$-stability of ODE solvers is related to the acceleration of the first-order optimization method \cite{luo2021differential}.


\subsection{The effect of the error of the inner solver}

Compared to the forward methods, backward methods obtain their next step by solving an implicit equation. Dense linear equations can be solved efficiently using exact methods like QR decomposition. In general, nonlinear equations, or even sparse linear equations, do not have the luxury. At best, an approximation to a solution bounded by a specific error tolerance is attained. Fortunately, a bound of linear growth can be established on the total effect caused by the inner solver error against per step tolerance. For example, we consider a one-dimensional equation $y' = f(y)$, and the backward Euler scheme with error can be written as
\begin{equation*}
  \frac{y^{k+1} - y^k}{s} = f(y^{k+1}) + \epsilon_k,
\end{equation*}
where $\epsilon_k\leq \epsilon$ is the error for the inner loop with $\epsilon>0$ being the tolerance. Let $\hat{y}^k$ be the exact backward Euler scheme solution, 
Then we have
\begin{equation*}
    C_k\lvert {y^{k+1} - \hat{y}^{k+1}} \rvert 
    \leq \lvert {y^{k} - \hat{y}^{k}} \rvert + \epsilon_k.
\end{equation*}
Where $C_k = \min_{\xi \in I(y^{k+1}, \hat{y}^{k+1})} \lvert 1-sf'(\xi)\rvert$, and $I(a,b)$ denote the interval $[\min(a,b), \max(a,b)]$. Hence
\begin{equation*}
  \lvert {y^{k} - \hat{y}^{k}} \rvert 
  \leq \sum_{j=0}^{k-1} \prod_{i=j}^{k-1} C_k^{-1} \epsilon_j 
  \leq \sum_{j=0}^{k-1} \prod_{i=j}^{k-1} C_k^{-1} \epsilon,
\end{equation*}
which provides a linear upper bound on the error. Similar arguments can be generalized to other solvers and for high-dimensional ODEs.

\subsection{Energy stability and convergence}

{Besides the linear stability, there is another notion of stability, known as {\bf energy stability}. We say a numerical method is energy stable if $F(\vh_{k+1}) \leq F(\vh_{k})$ for all $k$ \cite{shen2019new}.}
As an advantage of the proximal formulation, one can show that the backward Euler is unconditionally energy stable \cite{xu2019stability} for arbitrary $s$ even for non-convex $F(\vz)$ (here we assume that there exists $F(\vz)$ such that ${\bm f} (\vz) = -\nabla F(\vz)$). 
The energy stability guarantees that the numerical solver takes each step toward 
minimizing $F(\vz)$, which further lead to the convergence of the numerical scheme \cite{wang2021particle}. 
More precisely, we have the following result (see Appendix~\ref{appendix:proofs} for the detailed proof).
\begin{proposition}\label{thm:converge}
If $F(\vz)$ is continuous, coercive and bounded from below, for any choice of $s > 0$, there exists $\vh_{k+1}$ solves \eqref{eq:backward-euler}, such that
\begin{equation}\label{eq:decrease}
    F(\vh_{k+1}) \leq F(\vh_{k}).
\end{equation}
Moreover, for non-convex $F(\vz)$, $\vh_{k+1}$ is unique provided $s \leq - 1/ \lambda_1$, where $\lambda_1 < 0$ is the smallest eigenvalue of $\nabla^2 F$.
\end{proposition}

Similar energy stability analysis holds for other type of proximal solvers. For instance,
for BDF2, 
one can show there exists $\vh_{k+1}$ such that
\begin{equation*}
    F(\vh_{k+1}) + \frac{ \| \vh_{k+1} - \vh_{k} \|^2}{4s} \leq F(\vh_k) + \frac{ \| \vh_k - \vh_{k-1} \|^2}{4s}.
\end{equation*}
For Crank-Nicolson, 
the proximal formulation \eqref{eq:prox-cn} leads to
\begin{equation*}
\begin{aligned}
 F(\vh_{k+1}) +  \langle \vh_{k+1} - {\vh}_{k}, \nabla F({\vh}_k) \rangle \leq F(\vh_{k}).
\end{aligned}
\end{equation*}
These energy stability result can further lead to the convergence of these scheme in the discrete level \cite{matthes2019variational}.

\section{Experimental Results}\label{sec:exp}
In this section, we validate the efficacy of proximal solvers and contrast their performances with benchmark adaptive neural ODE solvers in solving the one-dimensional diffusion problem and learning continuous normalizing flows (CNFs) 
\cite{grathwohl2018scalable} 
and GRAND \cite{pmlr-v139-chamberlain21a}. 


\subsection{Solving one-dimensional diffusion equation}
\label{subsec:diffusion-1D}



We consider solving the following 1D diffusion equation 
$$
\frac{\partial h}{\partial t} = \frac{\partial^2 h}{\partial x^2}
$$ 
for $t\in [0,1]$ and $x\in [0,1]$ with a periodic boundary condition. We initialize $u(x,0)$ with the standard Gaussian. 
To solve the diffusion equation, we first partition $[0,1]$ into uniform grids $\{x_i\}_{i=1}^{128}$ and discretize $\partial^2h/\partial x^2$ using the central finite difference scheme, resulting in the following coupled ODEs
\begin{equation}\label{eq:diffusion-discrete}
\frac{d\vh}{dt} = -\mL \vh,
\end{equation}
where $\mL\in\RR^{128\times 128}$ is the Laplacian matrix of a cyclic graph scaled by $1/\Delta x^2$ with $\Delta x=1/127$ being the spatial resolution, and $\vh=[h(x_1),\ldots,h(x_{128})]$. Then we apply different ODE solvers to solve \eqref{eq:diffusion-discrete}. It is worth noting that the diffusion equation has an infinite domain of dependence, thus in method of lines discretization, the time step to be taken for explicit methods has to be much finer compared to the spatial discretization to guarantee numerical stability.


We compare our proximal methods against the predominantly 
used Adaptive Heun and DOPRI5 in solving \eqref{eq:diffusion-discrete}. 
As the matrix $\mL$ is circulant, we can compute the exact solution of \eqref{eq:diffusion-discrete} using the discrete Fourier transform for comparison. Figure~\ref{fig:Diffusison-comparison} compares the NFEs of proximal methods and two benchmark adaptive solvers with different errors at the final time. For each final error, we use a fixed step size for each proximal solver with FR 
for solving the inner minimization problem. A list of the configuration of each ODE solver can be found in Table~\ref{tab:laplace_config} in the appendix. Figure~\ref{fig:Diffusison-comparison} shows that adaptive solvers require 
many more NFEs than proximal solvers. NFEs required by explicit solvers are almost independent of the error, revealing that the step sizes are constrained by numerical stability. Each iteration of higher-order BDF schemes requires more NFEs than lower-order BDF schemes, showing a tradeoff between controlling numerical error and using high-order schemes.

\begin{figure}[!ht]
\centering
\begin{tabular}{c}
\includegraphics[clip, trim=0.01cm 0.01cm 0.01cm 0.01cm, width=0.75\columnwidth]{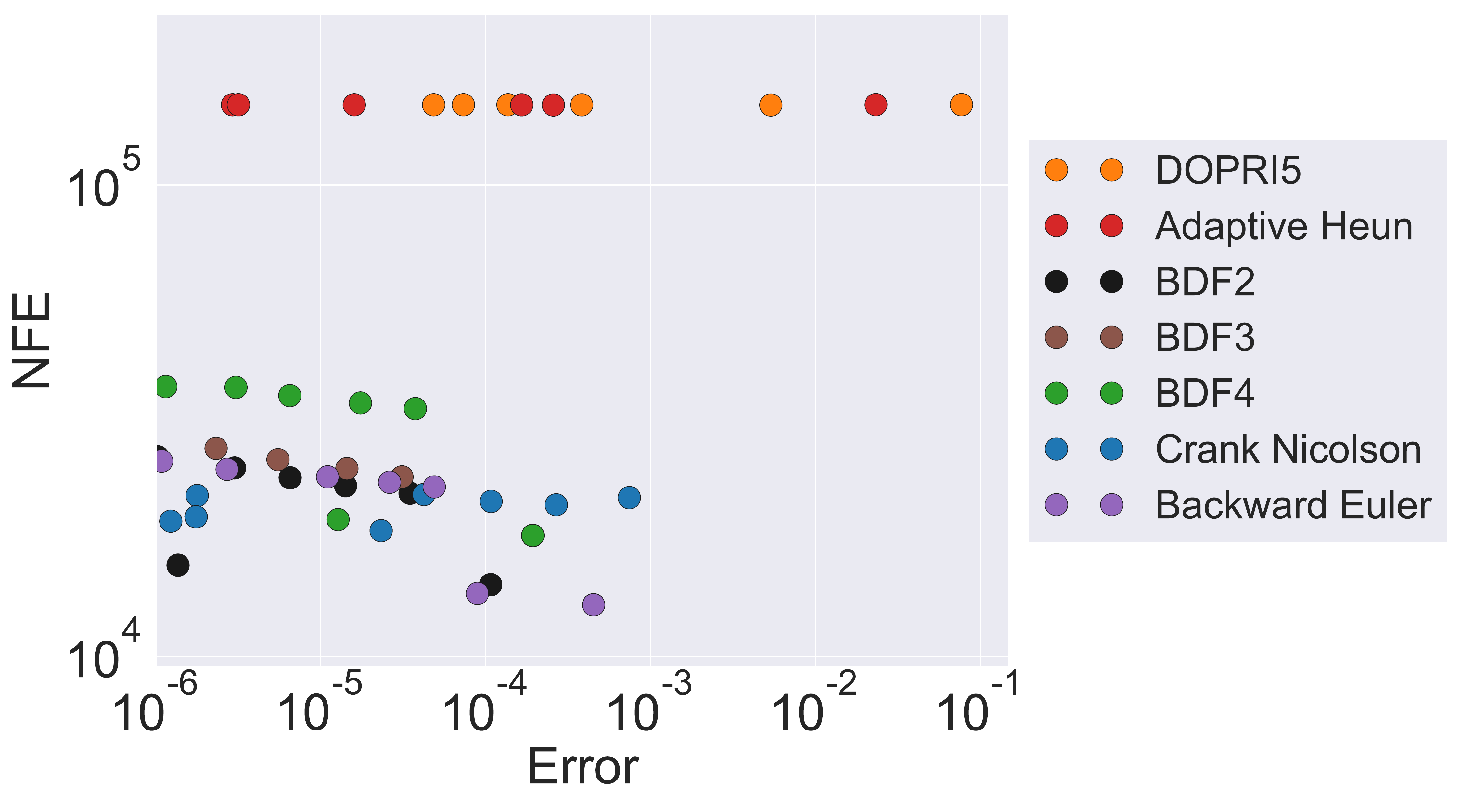}\\
\end{tabular}
\caption{Final step error vs. NFEs of different solvers in solving the one-dimensional diffusion equation \eqref{eq:diffusion-discrete}. Adaptive solvers require many more NFEs than proximal solvers, and NFEs required by explicit solvers are almost independent of the error since the step sizes here are constrained by numerical stability.}
\label{fig:Diffusison-comparison}
\end{figure}

\paragraph{Convergence of inner solvers.} 
An inner optimization algorithm is required for the proximal solver. We compare the efficiency of a few gradient-based optimization algorithms, including GD, NAG, Restart, and FR, 
for one pass of the proximal backward Euler solver. 
The comparison of different optimizers using the same step size $0.1$ (Figure~\ref{fig:compare-inner-optimizer}) shows that different solvers perform similarly to each other at the beginning of iterations, while as the iteration goes, FR performs best among the five considered solvers. 
The performance of different optimization algorithms further shows the rationale for solving the inner problem using gradient-based optimization algorithms since the convergence is consistent with the desired performance of each optimizer. We will use FR as the default inner solver for all experiments below. 
\begin{figure}[!ht] 
\centering
\begin{tabular}{cc}
\includegraphics[width=0.48\linewidth]{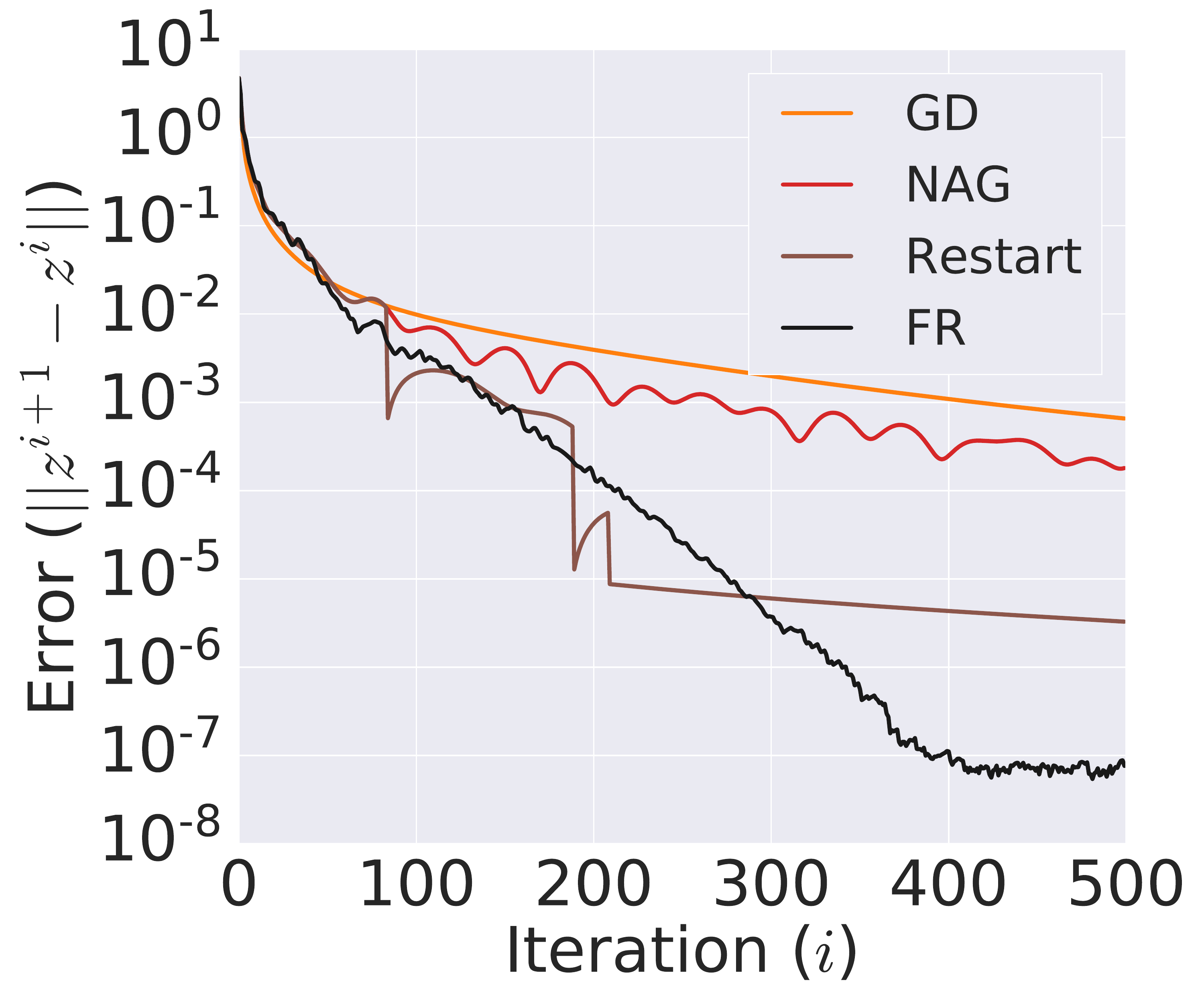}&
\includegraphics[width=0.48\linewidth]{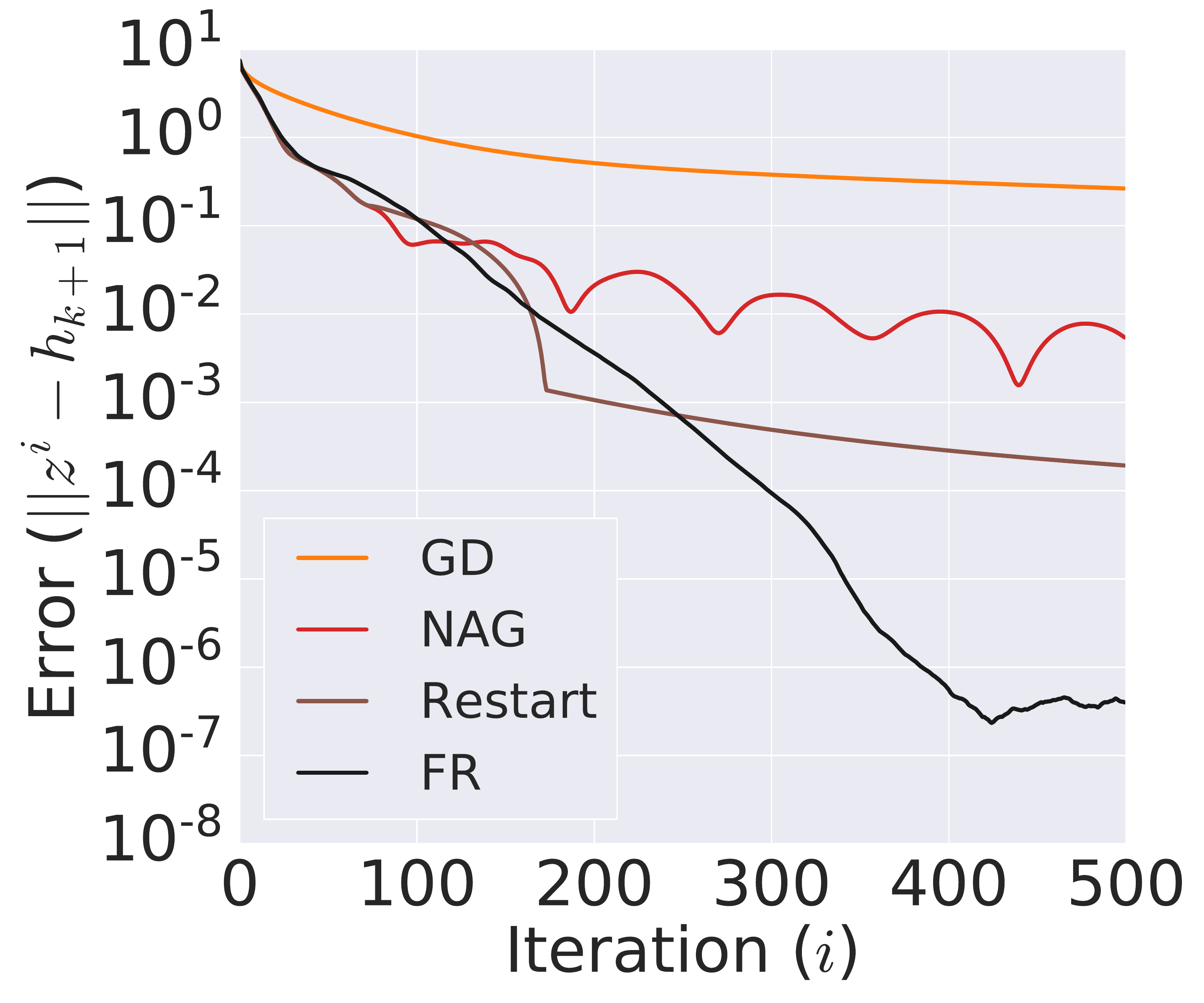}\\
(a) & (b) \\
\end{tabular}
\caption{Convergence comparison of different inner solvers for the proximal Backward Euler. (a) Convergence of $\|\vz^{i+1}-\vz^i\|$, and (b) convergence of $\|\vz^i-\vh_{k+1}\|$ for $k=0$. 
}\label{fig:compare-inner-optimizer}
\end{figure}

\paragraph{Proximal vs. nonlinear solvers.} We further verify the computational advantages of proximal solvers over other nonlinear root-finding algorithms for directly solving the implicit ODE solvers. In particular, we consider solving the 1D diffusion equation mentioned before using backward Euler solver given by \eqref{eq:backward-euler}, where we use proximal backward Euler and directly solve \eqref{eq:backward-euler} using either fixed point (FP) iteration or the Newton-Raphson method (NR). Figure~\ref{fig:1D-diffusion} (a) below shows the time required by different solvers when the spatial interval $[0,1]$ is discretized into different number of grids, controlling the scale of the problem. Also, the computational time of proximal solver is almost linear w.r.t. NFE (Fig.~\ref{fig:1D-diffusion} (b)). Proximal solver beats FP and NR. NR is not scalable to high-dimension since it requires evaluating Jacobian, which is very expensive in computational time and memory costs. 

\begin{figure}[!ht]
\centering
\begin{tabular}{cc}
\includegraphics[clip, trim=0.01cm 0.01cm 0.01cm 0.01cm, width=0.48\columnwidth]{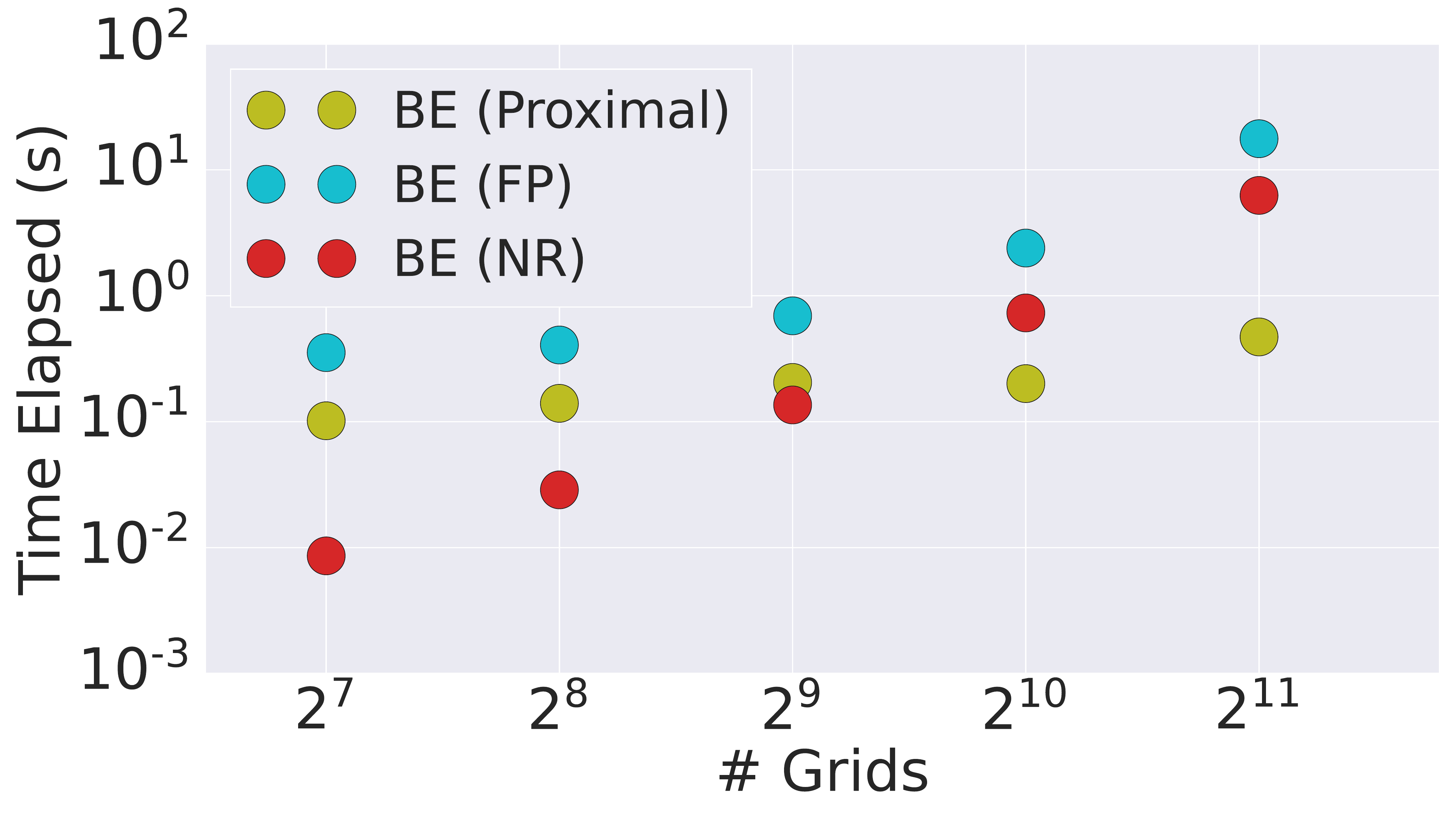}&
\includegraphics[clip, trim=0.01cm 0.01cm 0.01cm 0.01cm, width=0.48\columnwidth]{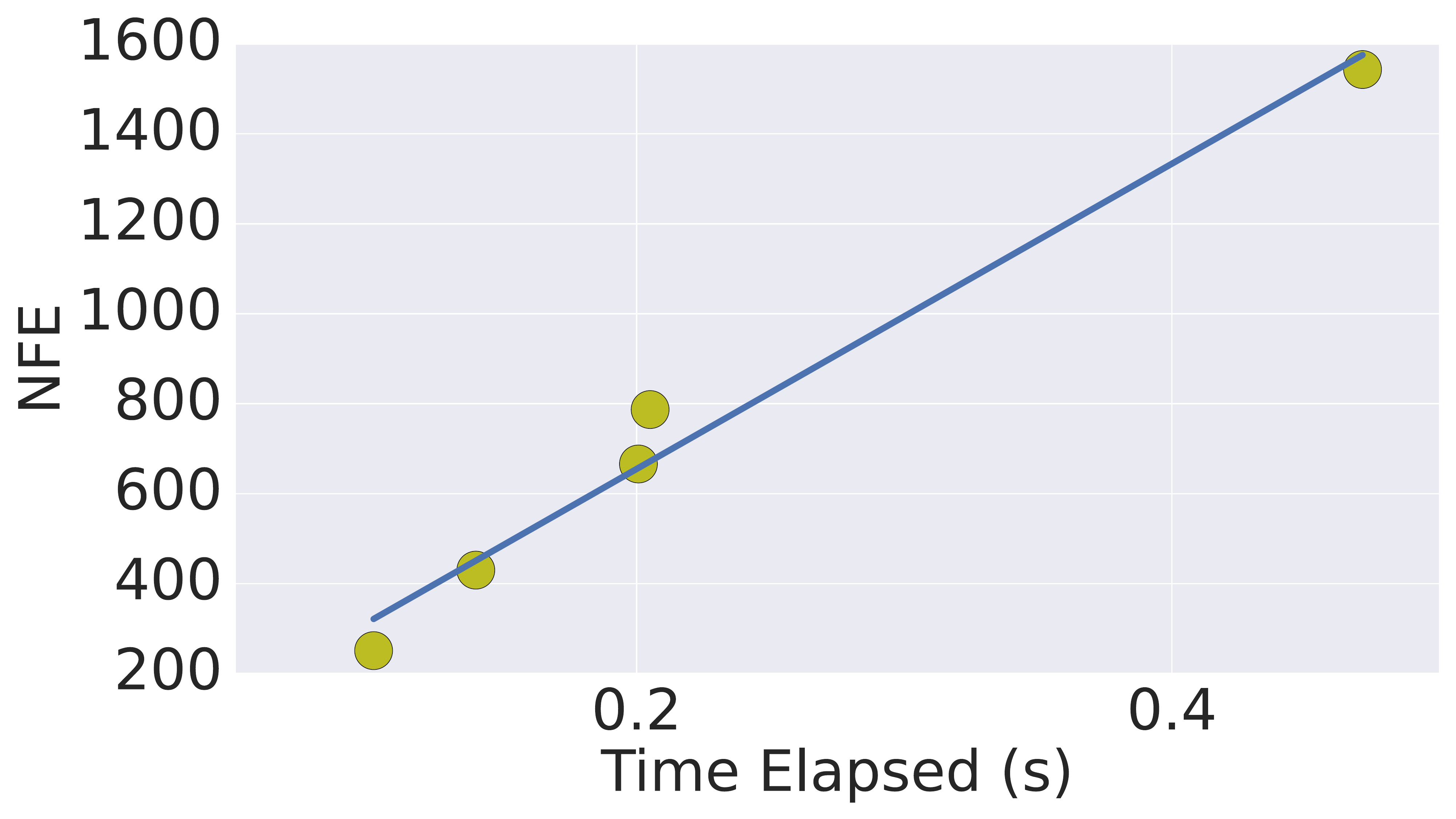}\\
{(a) Number of grids vs. time} & {(b) NFE vs. time}\\
\end{tabular}
\caption{(a) Computational time of solving 1D diffusion equation with different number of grids 
by backward Euler (BE) using proximal, FP, and NR solvers. (b) Computational time vs. NFE of proximal backward Euler solver for solving 1D diffusion equation.
}
\label{fig:1D-diffusion}
\end{figure}

{\it 2)} 
NFE is the number of function calls during integration; each optimization step takes one call, 
the reported NFEs of the proximal solvers in Sec.~4 equals the product of inner and outer iterations.
We solve inner problem using Fletcher-Reeves (see Sec.~4.1), which is much faster than GD.

{\it 3)} First, the $A$-stable domain determines the largest step size that the algorithm can take. The linear stability analysis shows that the implicit solver can use much larger step sizes than explicit solvers with numerical stability guarantee. Second, regarding why our theory shows implicit solver is ``faster'' than explicit solver: Notice that the step size used for solving ODEs in a given time interval is constrained by local truncation error and numerical stability, and the later is often the dominating factor, as illustrated in Fig.~1 and Sec. 1.1 in our paper. Since the implicit solver uses a much larger step size than explicit solvers with stability guarantee, the implicit solver takes fewer NFEs than explicit solvers.

















\subsection{Learning CNFs}

\begin{figure}[t!]
\centering
\begin{tabular}{ccc}
\includegraphics[clip, trim=0.01cm 0.01cm 0.01cm 0.01cm, width=0.27\columnwidth]{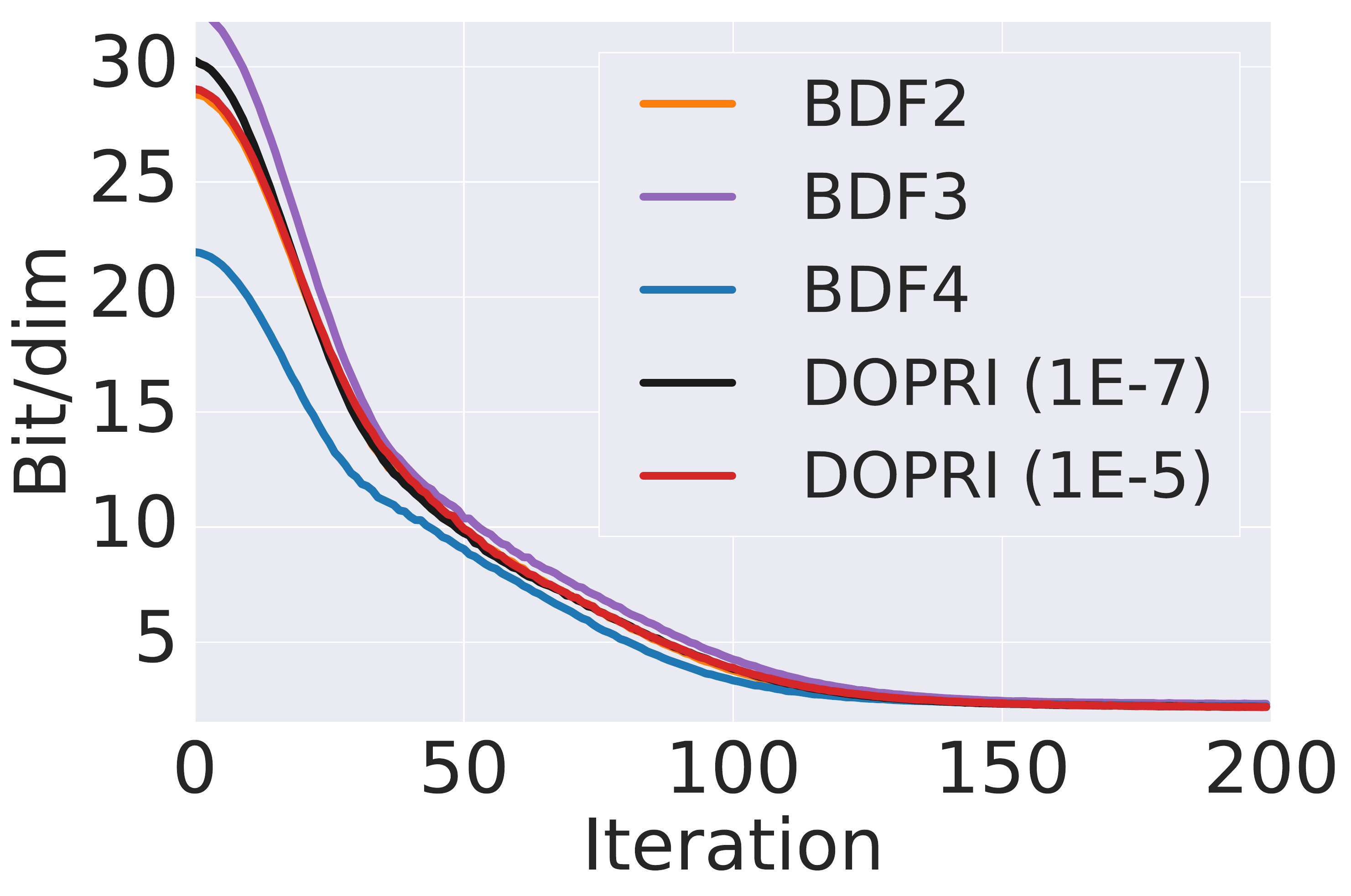}&
\includegraphics[clip, trim=0.01cm 0.01cm 0.01cm 0.01cm, width=0.27\columnwidth]{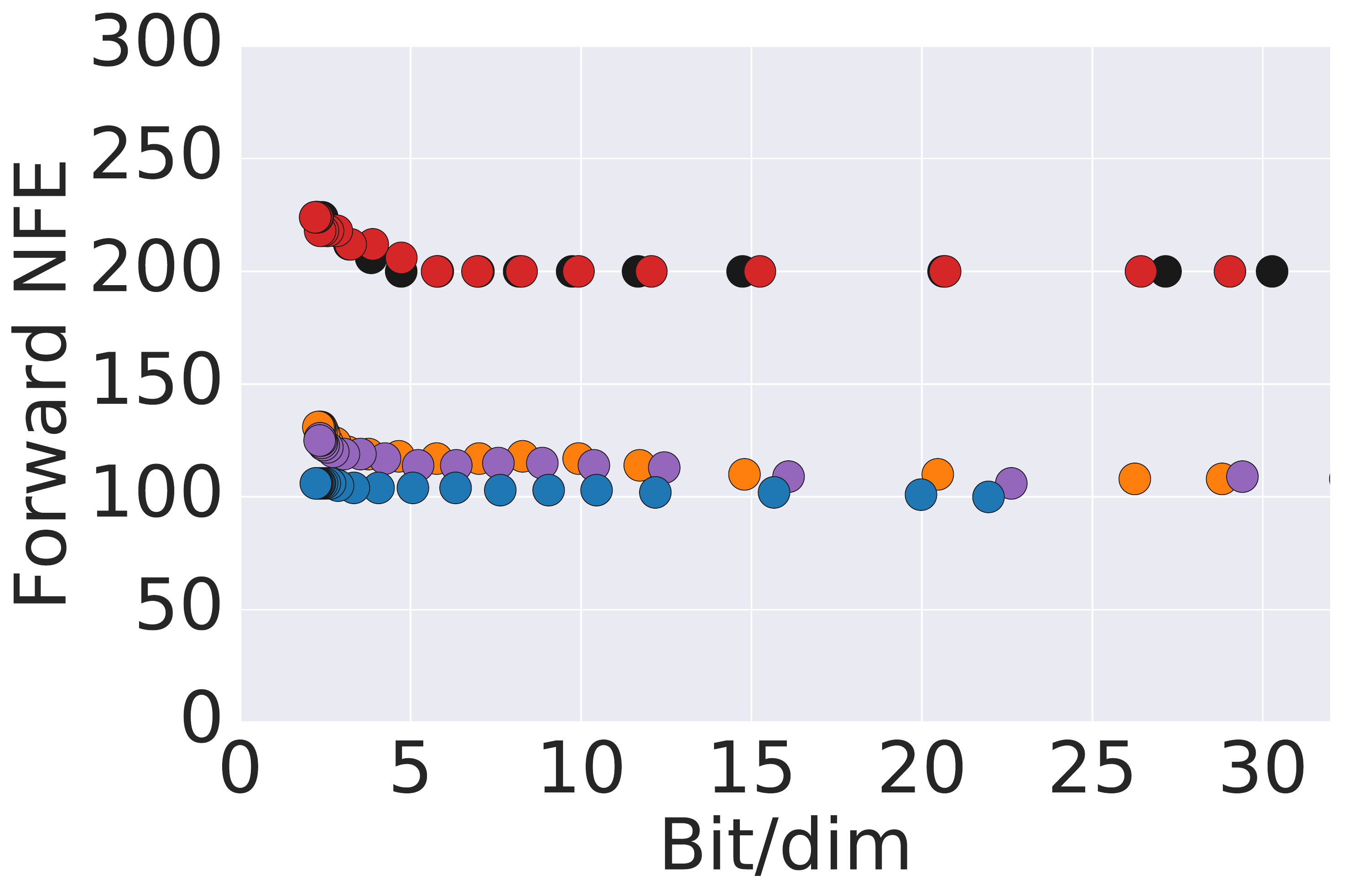}&
\includegraphics[clip, trim=0.01cm 0.01cm 0.01cm 0.01cm, width=0.4\columnwidth]{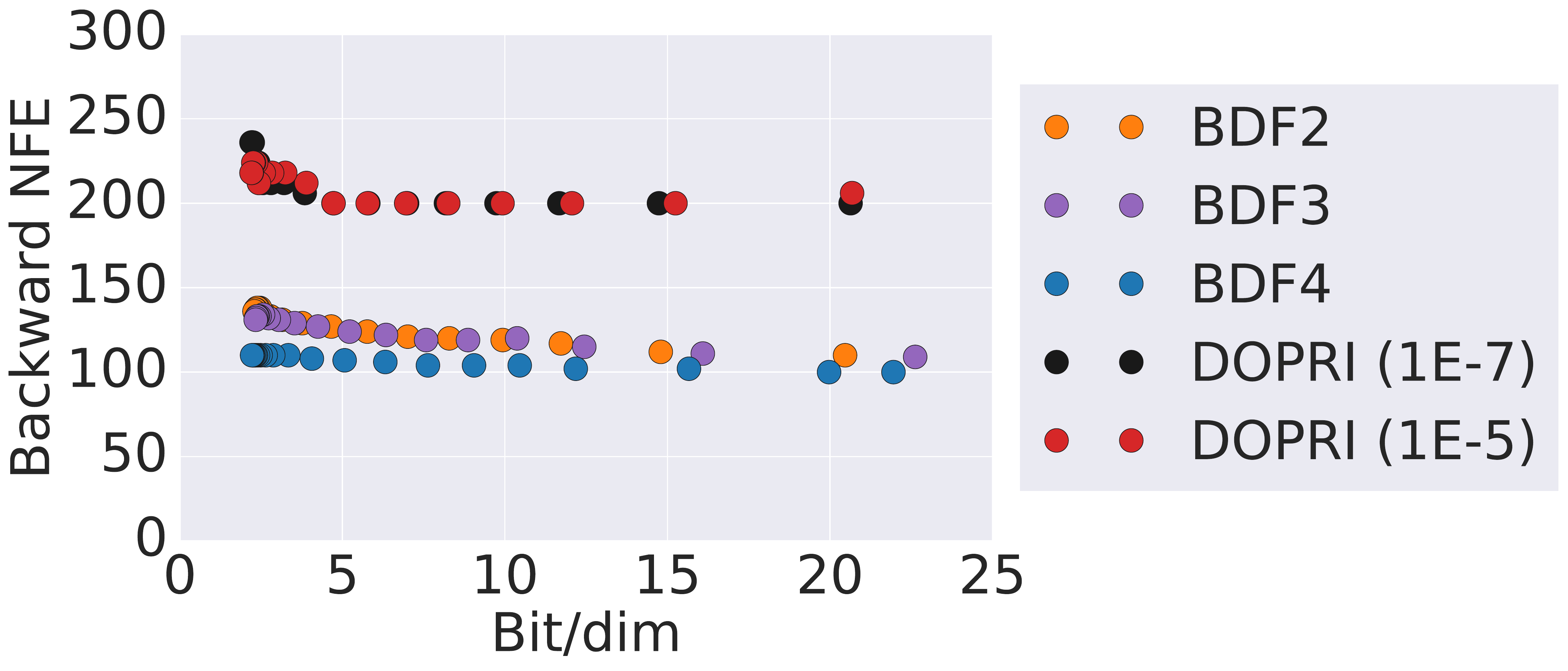}\\[-3pt]
(a) & (b) & (c)\\
\end{tabular}
\caption{Contrasting BDFs with DOPRI5 using different error tolerances for training CNFs for MNIST image generation. BDFs converge as well as DOPRI5 using very small error tolerances but require much fewer NFEs in solving both neural ODE and its adjoint ODE.
}
\label{fig:cnf_err_nfe}
\end{figure}

\begin{figure}[t!]
\centering
\begin{tabular}{ccc}
\includegraphics[clip, trim=0.01cm 0.01cm 0.01cm 0.01cm, width=0.42\columnwidth]{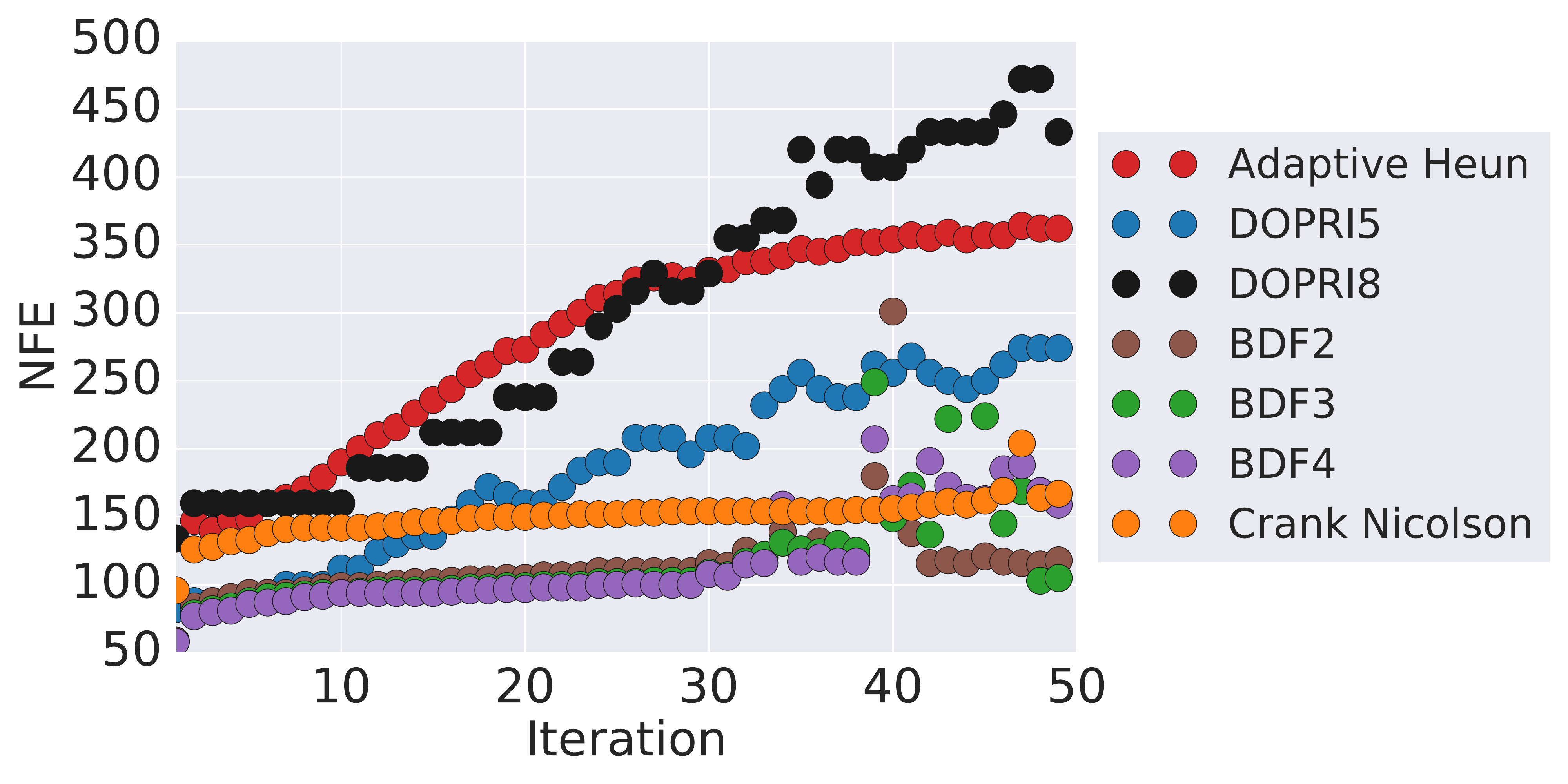}&
\includegraphics[clip, trim=0.01cm 0.01cm 0.01cm 0.01cm, width=0.26\columnwidth]{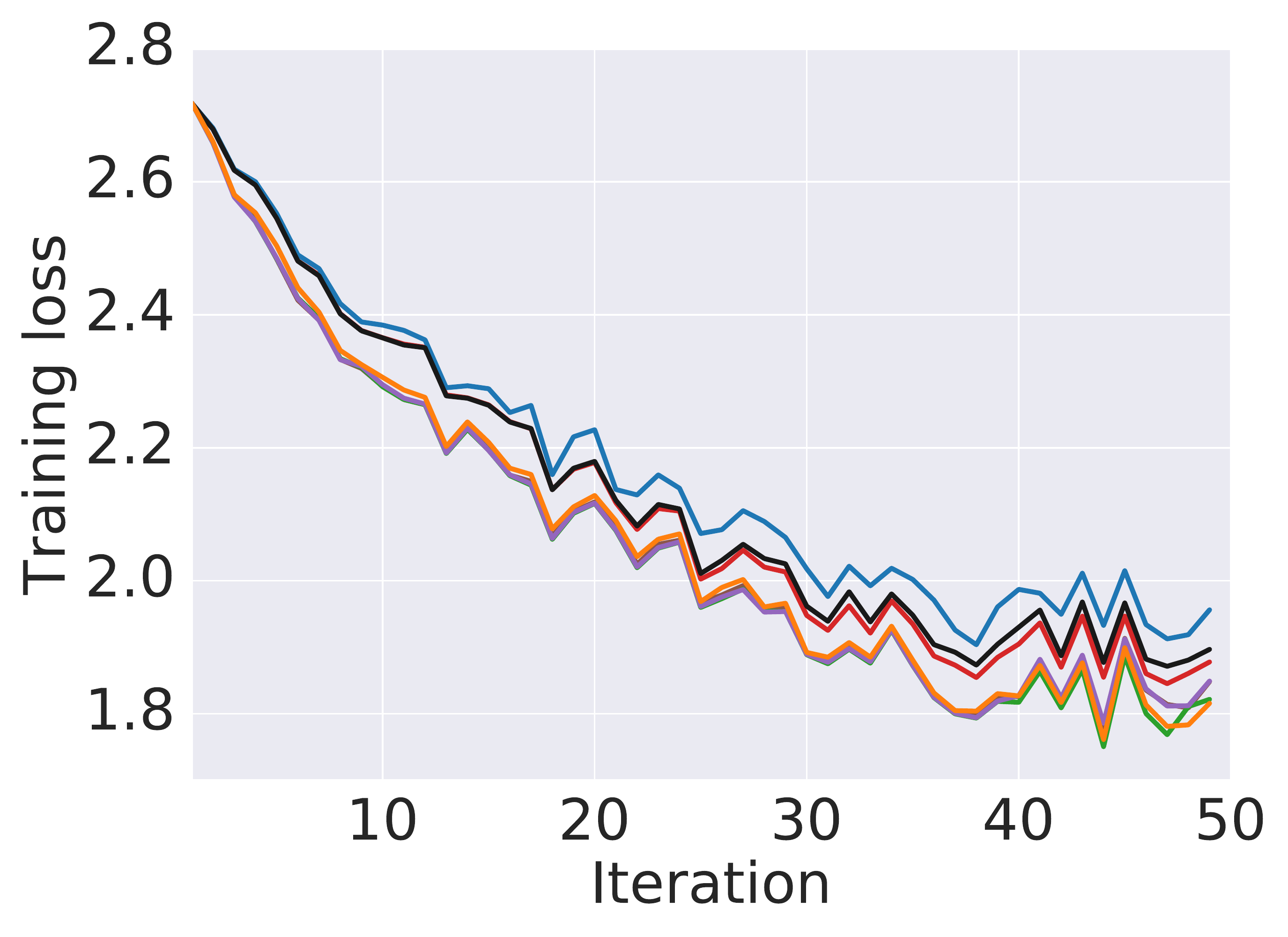}&
\includegraphics[clip, trim=0.01cm 0.01cm 0.01cm 0.01cm, width=0.26\columnwidth]{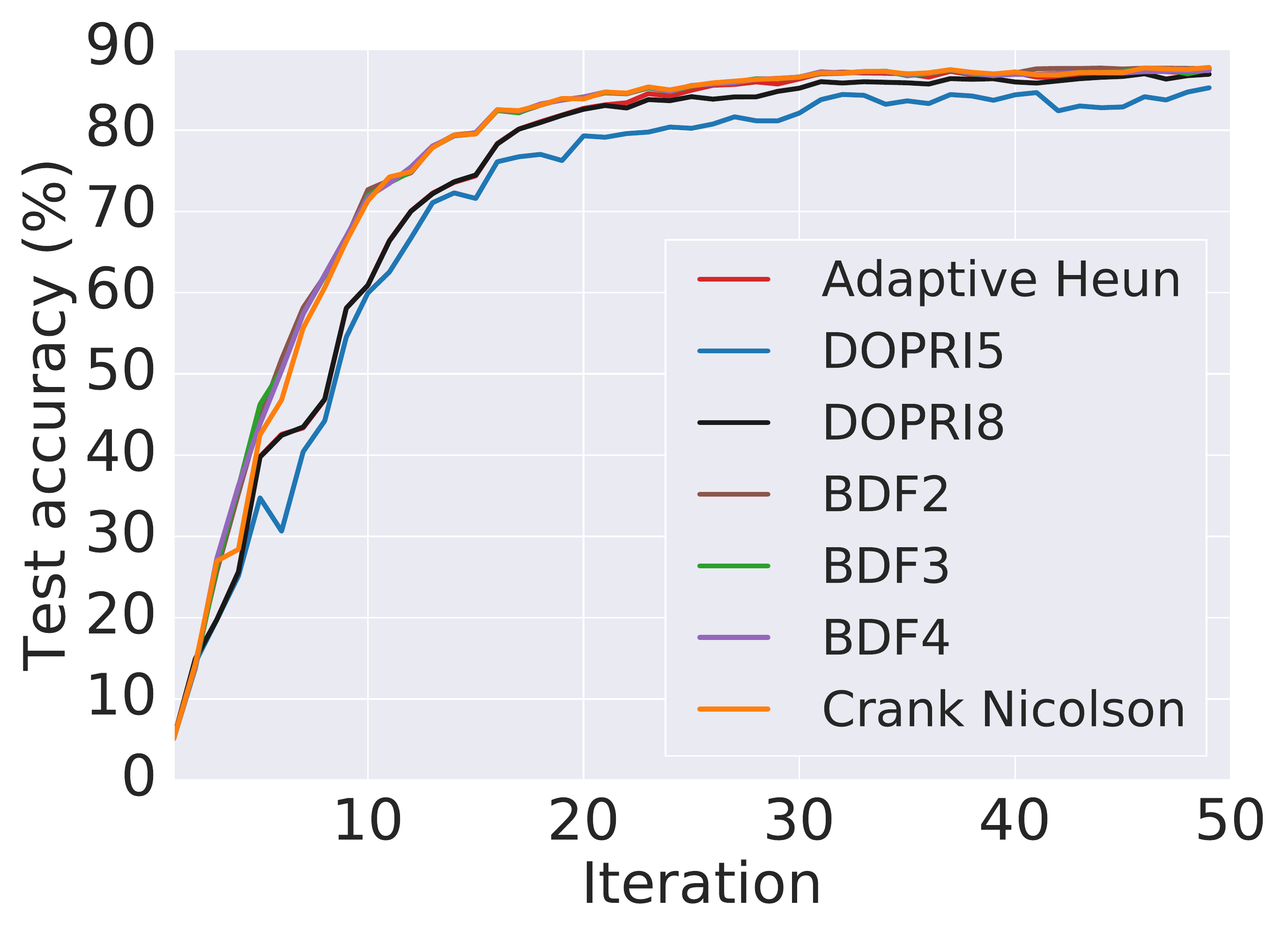}\\
(a) & (b)  & (c) \\
\end{tabular}
\caption{Comparison of proximal solvers with adaptive solvers in training GRAND for the CoauthorCS node classification. Proximal solvers require much fewer NFEs than adaptive solvers but converges as well as adaptive solvers in training loss and validation accuracy.
}
\label{fig:grand-comparison}
\end{figure}

We train CNFs for MNIST \cite{lecun-mnisthandwrittendigit-2010} generation using proximal solvers and contrast them to adaptive solvers. In particular, we use the FFJORD approach outlined in \cite{grathwohl2018scalable}. We use an architecture of a multi-scale encoder and two CNF blocks, containing convolutional layers of dimension $64\rightarrow64\rightarrow64\rightarrow64$ with a uniform stride length of 1 and a softplus activation function. We use two scale blocks and train for 10 epochs using Adam optimizer \cite{kingma2014adam} with an initial learning rate of $0.001$.
We run the model using the proximal BDFs with FR and DOPRI5 with relative tolerance of $10^{-5}$ and $10^{-7}$, respectively. All the other experimental settings are adapted from \cite{grathwohl2018scalable}. 
The proximal BDFs are tuned to have an error profile which was consistent with the DOPRI5 as shown in Figure~\ref{fig:cnf_err_nfe} (a). We choose $t_0=0$ and $T=1$ and the step sizes for numerical integration and the inner optimizer are 0.2 and 0.3, respectively. We terminate the inner optimization solver once $\|\vz^{i+1}-\vz^i\|\leq 2\times 10^{-3}$. A comparison of the forward and backward NFEs are depicted in Figure~\ref{fig:cnf_err_nfe} (b) and (c), respectively. We see that proximal BDF2, BDF3, and BDF4 require approximately half of the NFEs as DOPRI5 to reach similar training curves. 

\subsection{Training GRAND}\label{subsec:GRAND}

GRAND is a continuous-depth graph neural network proposed in \cite{pmlr-v139-chamberlain21a}. It consists of a dense embedding layer, a diffusion layer, and a dense output layer, with ReLU activation function in between. The diffusion layer takes $\vu(t_0)$ as input and solves \eqref{eq:GRAND-int2} below
to time $T$
\begin{equation}\label{eq:GRAND-int2}
    \frac{d \vu(t)}{d t} = \mA(\vu, \theta) \vu - \vu,
\end{equation}
where $\mA$ is the graph attention function \cite{velickovic2018graph} with learnable parameter set $\theta$. 
This ODE to be solved in GRAND is nonlinear and diffusive, which is particularly challenging to solve numerically.

We compare the performance of proximal solvers with three major adaptive ODE solver options (DOPRI5, DOPRI8, Adaptive Heun) on training the GRAND model for CoauthorCS node classification. The CoauthorCS graph contains 18333 nodes and 81894 edges, and each node is represented by a 6805 dimensional vector. The graph nodes contains 15 
classes, and we aim to classify the unlabelled nodes. 
We choose $t_0=0$ and $T=10$ as our starting and ending times of the ODE 
\eqref{eq:GRAND-int2}. 
For each solver, we run the experiment for 50 epochs and record the average NFEs per epoch and accuracy of trained models. For adaptive solvers we choose tolerance $10^{-4}$, and for proximal solvers we choose the integration step size 1.25 (8 steps in total) so that for the first iteration with same input admits output with error within $10^{-3}$. For optimization method we use Fletcher-Reeves with learning rate 0.3, and stop when error is below $10^{-4}$ tolerance. All the other experimental settings are adapted from \cite{pmlr-v139-chamberlain21a}. Figure~\ref{fig:grand-comparison} shows the advantage of proximal methods over adaptive solvers in training efficiency and maintaining the convergence of the training.





\section{Concluding Remarks}\label{sec:conclusion}
We propose accelerating learning neural ODEs using proximal solvers, including proximal BDFs and proximal Crank-Nicolson. These proximal schemes approach the corresponding implicit schemes as the accuracy of the inner solver enhances. Compared to the existing adaptive explicit ODE solvers, the proximal solvers are better suited for solving stiff ODEs for numerical stability guarantees. 
We validate the efficiency of proximal solvers on learning benchmark graph neural diffusion and continuous normalizing flows. There are several interesting future directions. In particular, 1) accelerating the inner solver from the fixed point iteration viewpoint, e.g., leveraging Anderson acceleration, and 2) developing adaptive step size proximal solvers by integrating proximal solvers of different orders.



\appendix

\section{Dormand-Prince Method: A Review
}\label{sec:explicit-solvers}

In this section, we briefly review the scheme, error control, and step size rule of the Dormand-Prince method, which is an explicit adaptive numerical ODE solver. 
The one step calculation, from $t_k$ to $t_{k+1}$ with step size $s$, in the Dormand-Prince method for solving \eqref{eq:NODE} is summarized below:
First, we update $\vh_k$ to $\vh_{k+1}$ using Runge-Kutta method of order 4.
$$
\begin{aligned}
\vk_1 &= s\vf(t_k,\vh_k)\\
\vk_2 &= s\vf\Big(t_k+\frac{1}{5}s,\vh_k+\frac{1}{5}\vk_1\Big)\\
\vk_3 &= s\vf\Big(t_k+\frac{3}{10}s,\vh_k+\frac{3}{40}\vk_1+\frac{9}{40}\vk_2\Big)\\
\vk_4 &= s\vf\Big(t_k+\frac{4}{5}s,\vh_k+\frac{44}{45}\vk_1-\frac{56}{15}\vk_2+\frac{32}{9}\vk_3\Big)\\
\vk_5 &= s\vf\Big(t_k+\frac{8}{9}s,\vh_k+\frac{19372}{6561}\vk_1-\frac{25360}{2187}\vk_2+\frac{64448}{6561}\vk_3-\frac{212}{729}\vk_4\Big)\\
\vk_6 &= s\vf\Big(t_k+s,\vh_k+\frac{9017}{3168}\vk_1-\frac{355}{33}\vk_2-\frac{46732}{5247}\vk_3+\frac{49}{176}\vk_4-\frac{5103}{18656}\vk_5\Big)\\
\vk_7 &= s\vf\Big(t_k+s,\vh_k+\frac{35}{384}\vk_1+\frac{500}{1113}\vk_3+\frac{125}{192}\vk_4-\frac{2187}{6784}\vk_5+\frac{11}{84}\vk_6\Big)
\end{aligned}
$$
And the $\vh_{k+1}$ is calculated as
$$
\vh_{k+1}=\vh_k+\frac{35}{384}\vk_1+\frac{500}{1113}\vk_3+\frac{500}{192}\vk_4-\frac{2187}{6784}\vk_5+\frac{11}{84}\vk_6.
$$
Second, we update $\vh_k$ to $\vh'_{k+1}$ by Runge-Kutta method of order 5 as
$$
\begin{aligned}
\vh_{k+1}'&=\vh_k+\frac{5179}{57600}\vk_1+\frac{7571}{16695}\vk_3+\frac{393}{640}\vk_4-\frac{92097}{339200}\vk_5 +\frac{187}{2100}\vk_6+\frac{1}{40}\vk_7.
\end{aligned}
$$
We consider $\|\vh'_{k+1}-\vh_{k+1}\|$ as the error in $\vh_{k+1}$, and given error tolerance $\epsilon$ we select the adaptive step size at this step to be
$$
s_{\rm opt}=s\Bigg(\frac{\epsilon s}{2\|\vh_{k+1}-\vh_{k+1}'\|} \Bigg)^{\frac{1}{5}}.
$$



Many other adaptive ODE solvers exist, and some are also used to learn neural ODEs, e.g., the adaptive Heun's method, which is a second-order ODE solver. More adaptive step solvers can be found at \cite{atkinson2011numerical}.



\section{Derivation of Proximal ODE Solvers}
In this section, we show the equivalence of the implicit solvers with their proximal formulation.

\begin{itemize}

\item Crank-Nicolson

$$
\begin{aligned}
&\frac{d}{d\vz} \left( \frac{\|{\vz} - {\vh}_k\|^2}{s} + F({\vz})+ \langle {\vz} - {\vh}_{k}, \nabla F({\vh}_k) \rangle \right) \Big|_{\vh_{k+1}}\\
=& 2 \left( \frac{\vh_{k+1} - \vh_k}{s} + \frac{1}{2} (\nabla F(\vh_{k+1}) + \nabla F(\vh_{k}) ) \right),
\end{aligned}
$$

\item BDF2: 
$$
\begin{aligned}
&\frac{d}{d\vz} \left( \frac{\|{\vz} - \vh_k\|^2}{s}  -  \frac{\|\vz - \vh_{k-1}\|^2}{4 s}+ F({\vz}) \right) \Big|_{\vh_{k+1}}\\
=&  \frac{1}{s} \Big( 2 (\vh_{k+1} - \vh_k) - \tfrac{1}{2} (\vh_{k+1} - \vh_{k-1}) \Big) + \nabla F(\vh_{k+1}),
\end{aligned}
$$

\item BDF3:
$$
\begin{aligned}
&\frac{d}{d\vz} \left( \frac{3 \|{\vz} - \vh_k\|^2}{2 s}  -  \frac{3 \|\vz - \vh_{k-1}\|^2}{4 s}  + \frac{\|\vz - \vh_{k-2}\|^2}{6 s} + F({\vz}) \right) \Big|_{\vh_{k+1}}\\
=& \frac{1}{s} \Big( \tfrac{11}{6} \vh_{k+1} - (3 \vh_{k} - \tfrac{3}{2} \vh_{k-1} + \tfrac{1}{3} \vh_{k-2} ) \Big) + \nabla F(\vh_{k+1}),
\end{aligned}
$$

\item BDF4:
$$
\begin{aligned}
&\frac{d}{d\vz} \left(  \frac{3 \|{\vz} - \vh_k\|^2}{2 s}  -  \frac{3 \|\vz - \vh_{k-1}\|^2}{4 s} + \frac{\|\vz - \vh_{k-2}\|^2}{6 s} + F({\vz} \right) \Big|_{\vh_{k+1}}\\
=& \frac{1}{s} \Big( \tfrac{25}{12} \vh_{k+1} - (4 \vh_{k} - 3 \vh_{k-1} + \tfrac{4}{3} \vh_{k-2} - \tfrac{1}{4} \vh_{k-3} ) \Big) + \nabla F(\vh_{k+1}).
\end{aligned}
$$

\end{itemize}

\section{Missing Proofs}\label{appendix:proofs}
\begin{proof}[Proof of Proposition~\ref{thm:converge}]
For any given $\vh_k$, recall the definition of $G(\vz)$ in \eqref{eq:inner-problem}.

If $F(\vz)$ is a convex function, clearly, $G(\vz)$ admits a unique minimizer. If $F(\vz)$ is not convex, denote $\lambda_1 \leq 0$ be the smallest eigenvalue of $\nabla^2 F$. Direct computation shows that
\begin{equation}
    \nabla^2 G(\vz) = \frac{1}{s} {\bm I} + \nabla^2 F(\vz),
\end{equation}
which indicates that $G(\vz)$ is a convex function if  $s \leq - 1/ \lambda_1$. 
Notice that $ \lim_{|\vz| \rightarrow \infty} G(\vz) = \infty$, $G(\vz)$ admits a unique minimizer in $\mathbb{R}^d$.

If $G(\vz)$ is not convex, we define $\mathcal{S} = \{ G(\vz) \leq G( \vh_k )  \}$. By the coerciveness and continuity of $F(\vz)$,  $\mathcal{S}$
is a non-empty, bounded, and closed set. Hence, $G(\vz)$ admits a minimizer (not unique) $\vz^*$ in $\mathcal{S}$. We can take $\vh_{k+1} = \vz^*$, then  
\begin{equation*}
\frac{1}{2 s} \| \vh_{k+1} - \vh_{k} \|^2 + F(\vh_{k+1}) \leq F(\vh_k), 
\end{equation*}
which gives us Equation \eqref{eq:decrease}. Hence, the scheme is unconditionally energy stable.

For the sequence $\{ \vh_k \}$, since
\begin{equation*}
\| \vh_{k} - \vh_{k-1} \|^2 \leq 2 s(F(\vh_{k-1}) - F(\vh_{k})),
\end{equation*}
we have
\begin{equation*}
\sum_{k=1}^N \| \vh_{k} - \vh_{k-1} \|^2 \leq 2 s (F(\vh_{0}) - F(\vh_{N})) \leq C,
\end{equation*}
for some constant $C$ that is independent with $N$. Hence
\begin{equation*}
\lim_{N \rightarrow\infty} \| \vh_{N} - \vh_{N-1} \| = 0,
\end{equation*}
which indicates the convergence of $\{ \vh_k \}$. Moreover, since
\begin{equation*}
\vh_{N} = \vh_{N-1} - s \nabla F(\vh_{N}),
\end{equation*}
we have
\begin{equation*}
\lim_{N \rightarrow \infty} \nabla F (\vh_N) = 0,
\end{equation*}
so $\{\vh_k\}$ converges to a stationary point of $F(\vz)$.
\end{proof}

\section{Some More Experimental Details}\label{appendix-Configurations}
\subsection{Configurations of solvers for solving the one-dimensional diffusion equation}


We list the numerical integration step size and the corresponding final step error of each ODE solver in Table~\ref{tab:laplace_config}. Here, we use FR optimizer to solve the inner optimization problem with step size $0.1$. We stop the inner solver when $\|\vz^{i+1}-\vz^i\| \leq 5\times 10^{-9}$.


\begin{table}[!ht]
\fontsize{10.0}{10.0}\selectfont
\centering
\begin{threeparttable}
\caption{The configuration of different ODE solvers for solving the 1D Diffusion equation in Section~\ref{subsec:diffusion-1D}. Here, we use GD with FR optimization to solve the inner optimization problem with the step size $0.1$. Figure \ref{fig:Diffusison-comparison} is generated by considering a range of inner optimization tolerances from $10^{-3}$ to $10^{-7}$. The final step error is the error between the numerical solution at $t=1$ and the exact solution. We report in the following figure the smallest final step error for each proximal algorithm.
}\label{tab:laplace_config}
\begin{tabular}{rccc}
\toprule[1.0pt]
Proximal algorithm & Numerical integration step size $s$ & Optimization step size $\eta$ & Final step error \cr
\midrule[0.8pt]
Backward Euler & $1/2000$ & $0.1$ & $4.68$e$-7$\\
Crank Nicolson & 
$1/2000$
& $0.1$ & $1.75$e$-6$\\
BDF2 & 
$1/2000$
& 
$0.1$
& $1.36$e$-6$ \\
BDF3 & $1/2000$ & $0.1$ & $5.71$e$-7$ \\
BDF4 & $1/2000$ & $0.1$ & $1.15$e$-6$ \\
\bottomrule[1.0pt]
\end{tabular}
\end{threeparttable}
\end{table}



\end{document}